\documentclass[10pt]{amsart}
\usepackage{amsmath, amssymb,amsfonts,amsthm,latexsym, amsxtra}
\usepackage{graphics, tabularx, multicol}
\usepackage{pstricks,pst-node,pst-plot, psboxit}
\usepackage{multido}
\usepackage{tikz}
\usepackage[all,2cell,dvips]{xy} \UseAllTwocells \SilentMatrices
\xyoption{curve}
\xyoption{import}     
\xyoption{arc}
\xyoption{ps}
\usepackage[latin1]{inputenc}
\usepackage[T1]{fontenc}
\newcommand{\BIGboxplus}{%
\mathop{%
\mathchoice%
{\raise-0.15em\hbox{\Large $\boxplus$}}
{\raise-0.15em\hbox{\large $\boxplus$}}
{\hbox{\large $\boxplus$}}
{\boxplus}
}
}
\newcommand{\A}{{\mathbb A}}

\newcommand{\C}{{\mathbb C}}

\newcommand{\G}{{\mathbb G}}

\newcommand{\Q}{{\mathbb Q}}
\newcommand{\R}{{\mathbb R}}

\newcommand{\V}{{\mathbb V}}
\newcommand{\Z}{{\mathbb Z}}

\newcommand{\calo}{{\mathcal O}}
\newcommand{\cald}{{\mathcal D}}

\newcommand{\cala}{{\mathcal A}}
\newcommand{\calb}{{\mathcal B}}

\newcommand{\cale}{{\mathcal E}}

\newcommand{\calh}{{\mathcal H}}
\newcommand{\cali}{{\mathcal I}}

\newcommand{\calj}{{\mathcal J}}
\newcommand{\call}{{\mathcal L}}

\newcommand{\cals}{{\mathcal S}}
\newcommand{\calr}{{\mathcal R}}

\newcommand{\dualg}{{\widehat G}}

\newcommand{\dualm}{{\widehat M}}
\newcommand{\dualn}{{\widehat N}}

\newcommand{\n}{{\mathfrak n}}

\newcommand{\pip}{\boldsymbol{\Pi}}

\newcommand{\GO}{\mathbf{G}}

\newcommand{\fp}{\mathfrak{p}}
\newcommand{\fnn}{\mathfrak{n}}

\newcommand{\son}{SO_{2n}}
\newcommand{\dson}{\widehat{SO}_{2n}}
\newcommand{\sonv}{SO_{2n,v}}
\newcommand{\dsonv}{\widehat{SO}_{2n,v}}
\newcommand{\gnn}{GL_{2n}}
\newcommand{\so}{SO_{8}}

\makeatletter
\def\Ddots{\mathinner{\mkern1mu\raise\p@
\vbox{\kern7\p@\hbox{.}}\mkern2mu
\raise4\p@\hbox{.}\mkern2mu\raise7\p@\hbox{.}\mkern1mu}}
\makeatother
\newtheorem{teo}{Theorem}[section]
\newtheorem{lema}[teo]{Lemma}
\newtheorem{prop}[teo]{Proposition}
\newtheorem{cor}[teo]{Corollary}

\theoremstyle{definition}
\newtheorem{defn}[teo]{Definition}

\newtheorem{obs}[teo]{Remark}

\theoremstyle{remark}

\newtheorem*{ack}{\bf Acknowledgements}

\begin{document}
\author{ Octavio Paniagua-Taboada}
\title{Some consequences of Arthur's conjectures for special orthogonal even groups}
\thanks{This research was partially financed by a fellowship of CONACYT, Mexico (registration number 175437) and by the DFG-grant RA 1370/2-1}
\begin{abstract}
In this paper we construct explicitly a  square integrable residual automorphic representation of the special orthogonal group $SO_{2n}$, through Eisenstein series. We show that this representation comes from an elliptic Arthur parameter $\psi$ and appears in the space $L^2(\son(\Q)\backslash\son(\A_\Q))$ with multiplicity one. Next, we consider parameters whose Hecke matrices, at the unramified places, have eigenvalues bigger (in absolute value), than those of the parameter constructed before. The main result is, that these parameters cannot be cuspidal. We establish bounds for the eigenvalues of Hecke operators, as consequences of Arthur's conjectures for $\son$. Next, we calculate the character and the twisted characters for the representations that we constructed. Finally, we find the composition of the global and local Arthur's packets associated to our parameter $\psi$. All the results in this paper are true if we replace $\Q$ by any number field $F$.
\end{abstract}
\maketitle
\tableofcontents


\section{Introduction}
A central problem in the theory of automorphic forms is the spectral decomposition of the right regular representation of $L^2\left(G(F)\backslash G(\A_F)\right)$, where $G$ is a connected reductive group defined over a number field $F$ and $\A_F$ is the  adele ring of $F$. It is known that this representation decomposes as
\begin{equation}\label{introeq1}
L^2\left(G(F)\backslash G(\A_F)\right)= L_d^2\left(G(F)\backslash G(\A_F)\right) \widehat{\oplus}\,\, L_{\text{cont}}^2\left(G(F)\backslash G(\A_F)\right)
\end{equation}
The discrete part decomposes further into the residual and the cuspidal part. James Arthur formulated in \cite{Ar} and \cite{Ar1} some very precise conjectures with the aim to describe the type of representations that should occur in the discrete spectrum. He proposes a special type of global and local parameters. Globally, these parameters are continuous homomorphisms 
\begin{equation}\label{ch3a1}
\psi: \call_F\times SL_2(\C) \to {}^LG
\end{equation}
where $\call_F$ is the conjectural Langlands group and $^LG$ is the $L$-group of $G$.  More precisely, we have conjugacy classes of these morphisms in $^LG$, with some additional conditions. By restriction of \eqref{ch3a1} to each place $v$, we have local parameters
\begin{equation}\label{ch3a2}
\psi_v: \call_{F_v}\times SL_2(\C) \to {}^LG_v,
\end{equation}
Here $\call_{F_v}$ is the local Langlands group and $^LG_v$ is the $L$-group of the extension of scalars $G_v=G\otimes F_v$. The aim of Arthur's conjectures is to describe the discrete spectrum of the classical groups from what is known for $GL_N$. Arthur reformulates the results of M\oe glin and Waldspurger (see \cite{MW2}) and gives a full conjectural description for classical groups (see \cite{Ar2}). In our work, we analyze some of the consequences of Arthur's conjectures for $\son$.

In {\sc section  $2$} we consider the split group $\son$ ($n$ is even). Next,  we construct a discrete Arthur's parameter, which appears in the residual spectrum. The construction is made through parabolic induction and we consider the Eisenstein series associated to the inducing representation 
\[
\boldsymbol{\Pi}: = \pi_1\otimes |det|^{s_1} \bigotimes \cdots \bigotimes \pi_m \otimes |det|^{s_m}
\] 
which is a representation of the Levi subgroup $M=\prod^m GL_2$ (here we set $m:= \frac{n}{2}$). We show that the automorphic form associated to our parameter appears with multiplicity one. None of the results in this section depend on Arthur's conjectures.

In {\sc section  $3$} we study some parameters of $\son$ having in their Hecke matrices, at the unramified places, eigenvalues bigger (in absolute value) than those of the parameter $\psi$. Here we use the Mackey theory for calculating the restriction of an induced representation. We show that these parameters are singular in a sense defined by Howe and Li. The consequence is that these parameters must be residual.

{\sc section  $4$} deals with  Hecke operators for $\son$. Here we establish some bounds, better than the bounds already known for the eingenvalues of Hecke operators. For calculating the estimates, we use the Satake transform of some characteristic functions. These estimates rely on Arthur's conjectures.

In {\sc section  $5$} we calculate the character and the twisted character for the parameter of {\sc section  $2$}. We consider $\son$ as a twisted endoscopic group of $GL_{2n}\rtimes \theta$, where $\theta$ is an outer automorphism. We calculate in detail the map called \emph{norm} between stable conjugacy classes in $\son$ and stable $\theta$-conjugacy classes in $GL_{2n}$. After the calculation of the characters, we show an identity of traces between this two groups, where the transfer factor appears in a very simple form.

In {\sc section  $6$} we show how the conjectural results of Arthur correspond to our results, for the parameter $\psi$ of {\sc section  $2$}. We study and describe the local objects defined by Arthur and we prove that the corresponding local representation is irreducible.

Finally, in {\sc section  $7$} we describe the global objects defined by Arthur and we give the composition of the global Arthur's packet. We show that the packet of $\psi$ cannot have cuspidal representations.  We finish with a result which  ameliorates the bounds in {\sc section  $4$}, these bounds are essentially optimal, and they depend on the results announced by  Arthur (see \cite{Ar2}).
\begin{ack}
This paper comprises the results of my Ph.D. thesis. The author would like to thank this Ph.D. advisor Laurent Clozel for his guidance and support, along all this work. Likewise, the author expresses his gratitude to Colette M\oe glin and Guy Henniart for very fruitful discussions. Finally, the author thanks Pablo Ramacher and George-August-Universit\"at G\"ottingen for their hospitality when writing this paper.
\end{ack}

\section{The construction for $SO_{2n}$}
\subsection{Notations and conventions}
The main object  in this paper is the special orthogonal group defined over $\Q$ and
associated to the quadratic form $x_1 x_{2n} +x_2 x_{2n-1}+ \cdots +
x_n x_{n+1}$.  The anti-diagonal matrix is naturally associated to this quadratic form. This group is split over $\Q$.  This Lie group has a Dynkin diagram of type $\cald_{n}$. We suppose that $n$ is even  and we set $m:=\frac{n}{2}$. We shall follow the method developed by Langlands to construct our representations through parabolic induction. We fixe a standard Levi subgroup given by choice of the natural Borel $B$ subgroup, defined over $\Q$. Write $B= TU$ where $T$ is a maximal split torus of $B$ and $U$ denotes the unipotent radical of $B$. 
     Let $P$ be 
the parabolic subgroup such that its Levi component $M$ is isomorphic to $m$ copies of  $GL_2$, generated by  the odd roots $\alpha_1, \alpha_3, \dots, \alpha_{n-1}$  in the Dynkin diagram (recall that $n$ is even).

\begin{center}
\begin{tikzpicture}
\filldraw[black] (0,0) circle (2pt)  (1,0) circle (2pt) (2.4,0) circle (0.4pt) (2.5,0) circle (0.4pt) (2.6,0) circle (0.4pt) (4,0) circle (2pt) (5,0.5) circle (2pt) (5,-0.5) circle (2pt);
\draw (0,-0.3) node{\tiny $\alpha_{1}$} (1,-0.3) node{\tiny $\alpha_{2}$} (4,-0.3) node{\tiny $\alpha_{n-2}$} (5, 0.22) node{\tiny $\alpha_{n-1}$} (5,-0.7) node{\tiny $\alpha_{n}$};
\draw (0,0) -- (1,0) --  (2,0);
\draw (3,0) -- (4,0) --  (5,0.5);
\draw (4,0) --  (5,-0.5);
\end{tikzpicture}\end{center}
Let $W_{\son}$, $W_M$ be the Weyl group of $\son$ and $M$, respectively.  We denote by $\Delta_M$ the roots whose restriction to the maximal torus
of the center of $M$ is not trivial. So $\Delta_M= \{X_i-X_j, 2X_m\mid 1\le i<j \le m\}$ and this set has exactly $m$ roots. We shall build a residual representation through Eisenstein series. Thus, in order to determine the poles of the Eisenstein series, we need to calculate the intertwining operators. Then, we define a set $W(M)$, given by
\begin{equation}\label{arteqwm}
W(M) = \{ w\in W_G \mid \text{the coset of $w$ in $W_{\son}/W_M$ has minimal length} 
\end{equation}
\[
\text{and $wMw^{-1}$ is a standard Levi subgroup} \}.
\]
 We have only one association coset, so that implies $wMw^{-1}= M$ and $W(M)$ is a true set of roots. The elements of $W(M)$ are generated by the $m$ elementary symmetries (see \cite{M-W} for the general case).

\subsection{Eisenstein series}
We shall construct for the Levi subgroup $M$ the following representation: for each  factor $GL_2$ we fix $\pi_i$  a cuspidal irreducible automorphic representation of $GL_2$. We twist these representations by the unramified character $|det|^{s_i}$, with $s_i \in \C$. Consider then the representation $\pip$ of $M$ 
\[
\pip: = \pi_1\otimes |det|^{s_1} \bigotimes \cdots \bigotimes \pi_m \otimes |det|^{s_m}
\]
  Next, we define the Eisenstein series depending on the holomorphic variables $\underline{s}=(s_1, \dots, s_m)$ for any function $\phi_{\pip}$ in the space of the representation $\pip$ as
\begin{equation}\label{eqse0}
E(\phi_{\pip}, \pip,s)(g)= \sum_{\gamma\in P(\Q)\backslash G(\Q)}\phi_{\pip}(\gamma g,s)
\end{equation}

It is known that this Eisenstein series converges in some positive hyperplane. We shall study the behavior of this Eisenstein series outside of the region of convergence. The poles of the Eisenstein series are given by those of the constant terms. And the constant terms can be determined in terms of intertwining operators. We have an intrinsic formula that allows us to fully understand the intertwining operators, given in terms of Langlands quotients.

\begin{defn}
Let $G$ be a reductive connected split group, defined over $\Q$. Let $M$ be a Levi subgroup of $G$. Assume we have a subgroup $M_\alpha \subset M$, such that $M_\alpha$ has rank $1$ in $M$. Let $T_M$ be the maximal torus of the center of $M$ and $\alpha : T_M \to \G_m$ be the restriction to $T_M$ of the only root of $G$ non-trivial on $T_M$. Let $\rho$ be a cuspidal representation of $M$ and let $r_\alpha$ be representation of the dual group $\dualm$ associated to $\rho$. We define, whenever it converges, the \emph{partial Langlands quotient}, $Q^S(\pi,\alpha)$, as
\begin{equation}\label{artlq1.0}
Q^S(\pi,\alpha):=\frac{L^S(\pi,r_\alpha)}{L^S(\pi, r_\alpha)[1]}
\end{equation}
\end{defn}
Now we have this well known proposition.

\begin{prop}\label{pcm}
Let $w\in W(M)$. The partial intertwining operator (for the unramified places) associated to it is given by:
\[
M^S(w,\pi)=\prod_{\substack{\alpha > 0 \\ w\alpha <0}}Q^S(\pi,\alpha), \qquad \alpha \in R_{ind}(T_M,\son),
\] 
where $R_{ind}(T_M,\son)$ denotes the indivisible roots of $\son$ with respect to $T_M$.
\end{prop}

With this proposition it is easy to find the poles of these $L$-functions. Now, we make the additional hypothesis that $\pi_1= \pi_2=\cdots = \pi_m:= \pi$, i.e. we assume that all the representations $\pi_i$ are the same and we denote this representation by $\pi$.
\begin{cor}\label{ccm}
For the special orthogonal (split) group $SO_{2n}$ (we set $m:= \frac{n}{2}$), with our data, $M$ the Levi subgroup defined above, $M_\alpha, \pi$, let $\underline{s}= (s_1,\dots, s_n)$. The formula for the partial intertwining operator is given by 
\[
 M^S(w,\underline{s})=
\prod_{\substack{i <j\\ w(\epsilon_i-\epsilon_j) <0}}\frac{L^S(\pi\otimes\pi^\vee, s_i-s_j)}{L^S(\pi\otimes\pi^\vee, s_i-s_j+1)}\times\prod_{\substack{i < j\\ w(\epsilon_i+\epsilon_j) <0}}\frac{L^S(\pi\otimes\pi, s_i+s_j)}{L^S(\pi\otimes\pi, s_i+s_j+1)}\times
\]
\[
\times\prod_{w(2\epsilon_i)<0}\frac{\zeta^S(2s_i)}{\zeta^S(1+2s_i)}
\]
\end{cor}
This result allows us to localize the poles of the intertwining operators at the point $s_0$. There will be exactly $m$ terms having a simple pole: the quotient of $L$-functions $\frac{L^S(\pi\otimes\pi^\vee, s_1-s_2)}{L^S(\pi\otimes\pi^\vee,1+ s_1-s_2)},\dots, \frac{L^S(\pi\otimes\pi^\vee, s_{m-1}-s_m)}{L^S(\pi\otimes\pi^\vee, s_{m-1}-s_m+1)}$ and $\frac{\zeta^S(2s_m)}{\zeta^S(1+2s_m)}$. From this, we  can modify the Eisenstein series and remove all the poles.
\begin{prop}
Consider the point $s_0= (\frac{n-1}{2},\dots,\frac{3}{2},\frac{1}{2})$. The modified Eisenstein series $E^*(\phi_{\pip},\underline{s}):= (s_1-s_1^0)(s_2-s_2^0)\cdots(s_m-s_m^0)E(\phi_{\pip},\underline{s})$ is holomorphic and non-zero in $s_0$. The only non-zero constant term of $E^*(\phi_{\pip},\underline{s})$ corresponds to the element $w_0= 1(-1,\dots,-1)$. 
\end{prop}

Once we have built this modified Eisenstein series, we need to determine if  it is square integrable. For this we calculate the cuspidal exponents. To the representation $\pip$ we can associate an element of the lie algebra $\mathfrak{a}_M^{*}$ called the real part, denoted by $Re\,\pip$. We use now a result of Langlands (see \cite{M-W} \S I.4.11 lemma, page 74) to determine that this Eisenstein series is square integrable.

We shall prove that the intertwining operators are holomorphic in the ramified and archimedian places, i.e. the poles of the Eisenstein series are global. Indeed, we have the following result.

\begin{prop}\label{hubunavez}
The intertwining operator for the representation $\pip$ of $\son$, is holomorphic and non-zero for the archimedian and unramified places.
\end{prop}

\begin{proof}
We have to distinguish between two cases. Suppose that the local representation $\pi_v$ is tempered. Let $\omega \in M(W)$ (\emph{cf}. equation \eqref{arteqwm}) and we set $M_v(w,\pip)$ for the local intertwining operator. The element $\omega$ admits a minimal non-unique factorization $w=w_{i_1}\cdots w_{i_k}$. Thus, the operator has also a decomposition $M_v(w,\pip)$ into $k$ factors. Each factor satisfies the conditions of the Langlands quotient theorem (\cite{L1}) either for the parabolic subgroup whose Levi subgroup is isomorphic to $\prod^m GL_2 \subset GL_2 \prod^{m-2} GL_2$, or for the parabolic subgroup $\prod^m GL_2 \subset SO_4 \prod^{m-1} GL_2$. In fact, we can consider only the intertwining operator restricted to the subgroup $GL_2\times GL_2 \subset GL_4$ or $GL_2\times GL_2 \subset GL_2\times SO_4$, since the intertwining operator is constant on the other copies of $GL_2$. So we can reduce the problem  to the case
\[
Ind_{GL_2\times GL_2}^{P}(\omega\otimes |\;|^s) 
\]
where $P$ is $SO_4\times GL_2$ or $GL_4$, $\omega$ is an unramified unitary character and $s$ real and $>0$. Silberger \cite{Sil} and Borel-Wallach \cite{B-W} in the $p$-adic case and Langlands in the real case show the convergence and the holomorphy of the intertwining operator. Moreover, that the intertwining operator is non-zero is a consequence of the above fact, since its image is isomorphic to Langlands quotient that is irreducible.

Suppose now that the local representation $\pi_v$ is non-tempered. This means that the representation is unramified modulo a twist by quadratic character. As in the previous case, we can consider the intertwining operator only on the Levi subgroup $M'$, where $M'= GL_4$ or $M'= SO_4\times GL_2$. Recall that the representation $\pi$ of $GL_2$  is self-dual. Then, this implies that its local $v$-component, the representation $\pi_v$, is actually an induced representation
\[
\pi_v=Ind_B^{GL_2}(|\,|^\sigma,|\;|^{-\sigma})
\]
where $B$ is the standard Borel subgroup of $GL_2$ and $\sigma = Re(s) < \frac{1}{2}$. Therefore, we can consider an induction in stages and the intertwining operator acting in this induced representation. It is an induced representation of $2m$ characters $\chi_1, \dots, \chi_{2m}$. Thus the intertwining operator is realized as
\[
M_v(w,\pip)\phi_{\pip}:Ind_{B_{M^\prime}}^{M^\prime}(\chi_1, \dots,\chi_{2m})\to Ind_{B_{M^\prime}}^{M^\prime}(w\chi_1, \dots,w\chi_{2m})
\]
Thus, for $\pi_v$ unramified, there exists a vector $e_v$ fixed by the maximal compact subgroup $K_{M^\prime}$. For this vector we can restrict the intertwining operator to $K_{M^\prime}/K_{M^\prime}\cap B_{M^\prime}$. In this case Langlands' calculations give a form of the intertwining operator as a quotient of $L$-functions $\frac{L(\sigma)}{L(\sigma+1)}$. Since $\sigma < \frac{1}{2}$, this quotient converges absolutely and thus the intertwining operator is holomorphic. Remark also that these $L$-functions are non-zero, so that the intertwining operator is non-zero. For an arbitrary function $\phi_{\pip}$ on the space of $\pip$ we have $|\phi_{\pip_v}|  \le C|e_v|$, so the convergence for the spherical vector $e_v$ implies the convergence in general for $\phi_{\pip}$.  For $\pi_v$ non-tempered and ramified, we have $\pi_v = \pi_0 \otimes \epsilon_v$ where $\epsilon_v$ is a quadratic ramified character and $\pi_0$ is an unramified representation. In that case, we have
\[
Ind_{(GL_2)^m}^{\son}(\pi_v\otimes s)= Ind_{(GL_2)^m}^{\son}(\pi_0\otimes \epsilon_v\otimes s)=Ind_{(GL_2)^m}^{\son}(\pi_0\otimes s)\bigotimes \epsilon_v
\]
So, by transport of structure, we have the holomorphy and non-vanishing of the operator since the action of $\epsilon_v$ does not affect the calculations.
In the infinite places the argument is the same, namely the Langlands quotient theorem in the tempered case. For the non-tempered case we will have a quotient of $\Gamma$ functions $\frac{\Gamma(s)}{\Gamma(s+\sigma)}$ which are holomorphic functions and non-zero  for $Re(s) > 0$.
\end{proof}

Next we prove the existence of a unique quotient for the global induced representation.
\begin{prop}\label{uniqquotient}
The global representation $I:=Ind_P^{\son}(\pip)$ has a unique  irreducible quotient.
\end{prop}

\begin{proof}
It is well known that the global representation $I:= Ind_P^{\son}(\pip)$ decomposes into a tensor product of local representations $I= \otimes_v I_v$. So, it suffices to show that each local factor has an irreducible quotient. We already know that if the local component $\pi_v$ of the inducing representation is tempered we have Langlands quotient theorem, so $I_v$ has an unique irreducible quotient. Thus, we have to deal only with the case where $\pi_v$ is non tempered. We can suppose then, that $\pi_v$ is  unramified by transport of structure. We shall verify that in that case, the representation $I_v$ is cyclic generated by its spherical vector  $e_v$. Then consider $e_v \otimes \phi^v$ where $\phi^v \in \otimes_{w\neq v} I_w$. Consider the image of $I_v\otimes \phi^v$ by the modified Eisenstein series $E^*$. This image is totally reducible, i.e. of the form $\oplus \rho_i$ with $\rho_i$ irreducible. Moreover, we have $I_v^{K_v}= \C e_v$, where
$K_v=\son(O_v)$. Thus, exactly one of the representations $\rho_i$ is spherical, say $\rho_0$ and by cyclicity, $E^*(I_v)= \rho_0$. So globally $E^*$ is zero over the Langlands submodules (kernel of the arrow to Langlands quotient) for $v$ tempered and over the submodules of Kato (kernel of the arrow to the spherical quotient). So, the image of $E^*$ is an irreducible module, which is a tensor product of the Langlands and Kato modules which are equal in the tempered case after a result of Labesse \cite{Lab}.
\end{proof}

We only need to verify that our induced representation $I$ satisfies the conditions of Kato's theorem. Kato defines for an unramified character $\chi$ the following coefficients
\[
c_\alpha(\chi)=\frac{1-\chi(a_\alpha) q^{-1}}{1-\chi(a_\alpha)}=\frac{e_\alpha(\chi)}{d_\alpha(\chi)}
\]
where $a_\alpha$ is the coroot associated to $\alpha$ and $q$ is the cardinality of the residual field. Likewise Kato defines the vector
\[
e(\chi)=\prod_{\alpha> 0}e_\alpha(\chi)
\]
The subgroup $W_\chi$ denotes the stabilizer of $\chi$ in the Weyl group $W$, and $W(\chi)$ is the subgroup 
\[
W(\chi)= \left< w_\alpha \mid d_{\alpha}(\chi)=0\quad and \quad\alpha> 0\right>
\]
Now, we have this theorem due to Kato, see \cite{Kat}.
\begin{teo}[Kato]\label{kat}
Let $\chi$ be an unramified character of a reductive group. Then,  
\begin{enumerate}
\item The principal series $I_{\chi}$ is irreducible if and only if
\begin{enumerate}
\item $e(\chi)e(\chi^{-1})\neq 0$.
\item $W_{\chi}= W(\chi)$.
\end{enumerate}
\item The spherical vector is cyclic in $I_\chi$ in and only if
\begin{enumerate}
\item $e(\chi)\neq 0$.
\item $W_\chi=W(\chi)$
\end{enumerate}
\end{enumerate}
\end{teo}

In our case, we can write  the character $\chi$ as $\chi= (\chi_1, \chi_2, \chi_3, \chi_4)$, $\chi_i(p^{-1})\in \C^\times$, so $\chi$ has the form 
\[
\chi= (p^{3/2}t, p^{3/2}t^{-1}, p^{1/2}t, p^{1/2}t^{-1})
\]
In this case the character is regular which means that $W_\chi=W(\chi)= 1$. So we need to check that $e(\chi)\neq 0$. This implies that for every $\alpha> 0$ we must have $\chi(a_\alpha)\neq p$. The positive roots in our case are $\{X_i-X_j,X_i+X_j\mid i< j\}$. If $\alpha= (X_i-X_j)$ we have $a_\alpha= (0,\dots,\underbrace{p}_{\text{$i$-th place}},\dots,
\underbrace{p^{-1}}_{\text{$j$-th place}},\dots, 0)$, and if $\alpha = (X_i+X_j)$ then  $a_\alpha= (0,\dots,\underbrace{p}_{\text{$i$-th place}},\dots,
\underbrace{p}_{\text{$j$-th place}},\dots, 0)$. Thus, the character $\chi$ can be written as 
\[
\chi(a_\alpha)= (|\;|^{3/2 + \sigma}, |\;|^{3/2 - \sigma}, |\;|^{1/2 + \sigma}, |\;|^{1/2 - \sigma})
\]
From this, it is immediate that $\chi(a_\alpha)\neq p$, and so, we can apply  Kato's theorem.

Thus, we have the main theorem of this section.
\begin{teo}\label{tchin2}
Let $\son$ be the split orthogonal group. Let $\pi$ be an automorphic cuspidal irreducible representation of $GL_2$. Given $\pi$, we build the representation $\pip$ of the Levi subgroup $M \simeq \displaystyle\prod^m GL_2 \subset SO_{2n}$ defined by $\pip:= \pi\otimes
 |\text{det}|^{\frac{n-1}{2}}\bigotimes \cdots \bigotimes \pi\otimes |\text{det}|^{\frac{1}{2}}$. Then,  the representation $Ind_{MN}^{SO_{2n}} (\pip\otimes 1)$ belongs to the residual spectrum and is given by an elliptic Arthur  parameter. Moreover,  any representation $\pi$ of $GL_2$ defines a unique irreducible submodule $E^*(Ind_P^G(\pip))$ in the space $\cala^2(SO_{2n})$, i.e., it appears with multiplicity one.
\end{teo}


\section{Singular discrete parameters}
The aim of this section is to show that Arthur's parameters having blocks of $sp_k$ for $k > n+1$, must be residual. The method was developed by Duke, Howe and Li (see \cite{DHL}). We shall study the restriction of the local induced representation
\[
Ind_{SO_{n+2}\times \G^{n-1}_m N}^{SO_{2n}}(\epsilon\otimes \bf{1}\otimes \chi)
\] 
to the subgroup $N_Q$, the unipotent radical  of the Siegel maximal parabolic subgroup $Q$ whose Levi component is isomorphic to $GL_n$. Here $\bf{1}$ means the trivial representation of $SO_{n+2}$. The representation $\epsilon \otimes \bf{1}$ is by hypothesis self-dual, so the character $\epsilon$ has order two and $\chi$ is a character of the torus $\G_m^{n-1}$. We suppose this character to be unitary for now. For the study of the restriction, we appeal to Mackey theory, namely, the following result.
\begin{teo}
Let $G_1$ and $G_2$ be two subgroups of a group $G$, regularly associated in $G$.  Let $\pi$ be a representation of $G_1$. For all $x\in G$ we set $G_x= G_2 \cap (x^{-1}G_1 x)$ and 
\[
V_x= Ind_{G_x}^{G_2}\left(\eta \to \pi (x\eta x^{-1})\right)
\]
So $V_x$ is determined modulo isomorphism by the double coset $\overline{x}$ of $x$. If $\mu$ is an admissible measure over the double quotient $G_1\backslash G / G_2$, then
\[
Ind_{G_1}^{G}(\pi)|_{G_2} \simeq \int^\oplus_{G_1\backslash G / G_2} V_{\overline {x}}d\mu(\overline{x})
\]
\end{teo}
We determine the double cosets $N_Q\backslash\son/ P$ having positive measure. We have a result of Berstein and Zelevinsky (see \cite{BZ2}), which shows that, we have only one open orbit and one closed orbit. This open orbit is the only one with positive measure. We shall show that we can take $Qw_0 P$ as a representative of this open orbit. Taking into account all the roots included in the Levi subgroups of $Q$ and $P$, we easily verify that the orbit $Qw_0 P$ and the group $SO_{2n}$ have the same dimension. Therefore, $Qw_0 P$ is, indeed, the open orbit. After writing $Q w_0 P$ as
\[
Q w_0 P=N_Q M_Q w_0 M_P N_P = N_Q  M_Q M_P
\overline{N}_P w_0
\]

The Mackey calculation leads us to the study of the restricted representation 
\[
x\to Ind_{N_Q\cap xPx^{-1}}^{N_Q}(^x\epsilon\otimes \chi)
\]
where $\epsilon\otimes\chi$ is a character of $P$. The exponential notation means conjugacy.

We recall that we shall consider only the open coset $Qw_0P$ for the calculation. If $x\in Qw_0P$ we can write $x=n_Q m_Q w_0 p$, with evident notations. Thus, we have
\[
x\equiv m_Q w_0 p\quad mod \,N_Q
\]
with some ambiguities, since this decomposition is not unique. We consider then the induced representation
\begin{equation}\label{eq4.0}
Ind_{N_Q\cap m_Qw_0P w_0 m_Q^{-1}}^{N_Q}(^y(\epsilon\otimes \chi)), \quad y= m_Q w_0
\end{equation}
We know that $m_Q$ normalizes $N_Q$ and this implies
\[
^{m_Q^{-1}}Ind\simeq Ind_{N_Q\cap w_0P
  w_0}^{N_Q}(^{w_0}(\epsilon\otimes \chi))
\]
The group $M_Q$ acts naturally on $N_Q = \Lambda^2(F^n)$ as the group $GL_n$ and $\widehat{N_Q}= \Lambda^2(F^{n*})$. Then a form $f \in \Lambda^2(F^{n*})$ is degenerate if and only if  $^mf$ is
degenerate. The same is true for the characters of groups. The problem is then reduced to the study of the  induced representation
\begin{equation}\label{emk4}
Ind_{N_Q\cap w_0 P w_0}^{N_Q}(^{w_0}(\epsilon\otimes\chi))
\end{equation}
Thus, we need to calculate 
\[
N_Q \cap w_0P w_0= N_Q \cap M_P \overline{N}_P= N_Q\cap M_P
\]
We characterize this intersection in terms of roots. It is easy to verify that the intersection of these   groups is generated by the roots in the set $B$, where $B$ is given by
\begin{equation}\label{pincheb}
B= \{ X_i+X_j \quad m\le i<j\le n \}
\end{equation}
We denote by  $N_B$ the unipotent group $\prod_{\alpha \in B}N_\alpha$. We know that $\epsilon$  is a character of $M_P= SO_{n+1}\times \G^{n-1}_m$, trivial on the derived subgroup $SO_{n+1}^+$ and then trivial on the unipotent group $N_Q$. Therefore the restriction of the induced representation becomes
\begin{equation}\label{eson15}
Ind_{N_B}^{N_Q}( \boldsymbol{1} )= m_1L^2(N_Q/N_B),
\end{equation}
where $m_1$ is the multiplicity of the trivial representation $\bf{1}$. Consequently, the continuous components of \eqref{eson15} are the characters of  $N_Q$ trivial on $N_B$, or equivalently, the linear forms of the Lie algebra $\mathfrak{n}_Q$ trivial on $\mathfrak{n}_B$.

\subsection{Singularity of forms on $\n_Q/\n_B$}
Howe defined a type of Fourier coefficients  (\cite{How}) for the symplectic groups using Siegel modular forms. J.S. Li in \cite{Li} defines, in general, Fourier coefficients for the classical groups. We follow next his definition. Let $V \simeq \Q^{2n}$ be a vector space and $X,Y \subset V$ maximal totally isotropic vector subspaces, being in duality for $\left<\,,\,\right>$, the bilinear form which defines $\son$. Since the quadratic form associated to $\son$ is $x_1x_{2n}+\cdots+
x_{n}x_{n+1}$, these totally isotropic subspaces are easy to describe: $X=(x_1, \dots x_n,
0,\dots,0)$ and $Y= (0,\dots,0,x_{n+1},\dots,x_{2n})$, with $x_i\in \Q$. Let $Q$ be the maximal parabolic subgroup of $\son$ which preserves $X$. The unipotent radical $N_Q$ is an abelian group, isomorphic to the vector space $B(Y)$, the space of the anti-symmetric bilinear forms on $Y$. The dual space of $B(Y)$ identifies naturally with $B(X)$. Suppose we have fixed an additive character $\psi$ of $\A_\Q/\Q$. Then, Pontriagin's dual of $N(\A)/N(\Q)$ is identified with $B(X)$.  Let  $\phi$ be a smooth function on $\call^2(\son(\Q)\backslash \son(\A))$. Let $\phi_T$ be the Fourier coefficient along  $N$, defined by 
\begin{equation}\label{coeffFou}
\phi_T(g)=\int_{N(\A)/N(\Q)}\phi(zg)\overline{\psi_T(z)}dz\qquad g\in \son(\A)
\end{equation}
We shall show that every (global) Fourier coefficient given by \eqref{coeffFou} is a degenerate element in $\widehat{(N(\Q)\backslash N(\A))}$.

According to previous remarks, we have
\[
Q=GL(Y)N_Q\simeq GL_n\, N_Q
\]
The subgroup $N_Q$ is unipotent and abelian, as we have already said. In fact, $N_Q$ consists of the matrices that have the form
\[
n(a) = \begin{pmatrix} 1_n & a\\ & 1_a \end{pmatrix},
\]
where $a$ is an alternating matrix of $n\times n$. Thus, when we pass to the Lie algebra, we have
\[
N_Q\simeq  \Lambda^2(V)\quad \text{and} \quad \widehat{N}_Q\simeq \Lambda^2(V^*)
\]

 Now, we can prove the local result analogous to the result of Duke, Howe and Li for $\son$.
\begin{prop}\label{sonmiodhl1}
Consider $\rho= Ind_{SO_{n+1}\times \G^{n-1}_m N}^{SO_{2n}}(\epsilon \otimes
{\bf 1}\otimes\chi)$. Let $\psi$ be a fixed non-trivial additive character of 
$\Q_p$. Let $\xi$ be a character of $N_Q(\Q_p)$ appearing in the spectral decomposition
of $\rho|_{N_Q}$. Write $\xi(x)= \psi((X,\lambda))$, where $x\in N_Q$, $x=exp(X)$, $X\in \mathfrak{n}_Q$
(the Lie algebra of $N_Q$) and $\lambda \in \mathfrak{n}_Q^*\simeq
\Lambda^2(V^*)$ (the space of alternating forms on $V$). Then
$\lambda$ is degenerate (we shall say in that case that $\xi$ is also degenerate).  
\end{prop}
\begin{proof}
The  totally isotropic space $V$ is constructed from the classical canonical basis $\left<e_1,\dots, e_n \right>$. If $\xi$ occurs in the spectral decomposition,  $\lambda$ is an element of $\widehat{\mathfrak{n}}_Q$ trivial on
$\mathfrak{n}_B$. According to the characterization of the set  $B$ defined by 
\eqref{pincheb}, the subspace $\mathfrak{n}_B$ is generated by
\[
\mathfrak{n}_B =\left<e_m\wedge e_{m+1} \dots, e_m\wedge e_n,
e_{m+1}\wedge e_{m+2},\dots, e_{n-1}\wedge e_{n}  \right>
\]
thus $\lambda$ is trivial on $e_i\wedge e_j$, with $m\le i<j\le n$. This implies that the linear form $\lambda$ can be written in the dual basis of $\Lambda^2(V)$ as an alternating matrix having a subblock consisting of zeros of size $(m+1)\times (m+1)$. Therefore the determinant of this matrix is zero and $\lambda$ is degenerate.
\end{proof}
This result has this important global consequence. 
\begin{cor}\label{sonclmio8}
A global representation of $\son$ having at the place $p$ the local representation $Ind_{SO_{n+2}\times \G^{m-1}_m N}^{SO_{2n}}(\epsilon \otimes{\bf 1}\otimes\chi)$, with $\chi$ unitary, is residual.
\end{cor}
For the proof of this corollary we need a result that Howe stated in the case of symplectic groups.
\begin{prop}\label{howelemma}
Let $\phi \in L^2(\son(\Q)\backslash \son(\A))$. The following statements are equivalent
\begin{enumerate}
\item The rank of $\phi$ is at most $k$.
\item In the representation by right translations of $\son(\A)$ on the subspace of $L^2(\son(\Q)\backslash \son(\A))$ generated by $\phi$, all  the local subgroups $\son(\Q_p)$ act by representations of rank $\le k$.
\item In the representations by right translations of $\son(\A)$ on the subspace of $L^2(\son(\Q)\backslash \son(\A))$ generated by $\phi$, there exists a prime number $p$ such that $\son(\Q_p)$ acts by representations of rank $\le k$.
\end{enumerate}
\end{prop}
\begin{proof}
The proof is essentially the same as given by Howe (see \cite{How}) with slight modifications.
\end{proof}

\begin{proof}[Proof of corollary \ref{sonclmio8}] 
If the representation $Ind_{SO_6\times \G_m N}^{SO_8}(\epsilon \otimes {\bf
1}\otimes \chi)$ was cuspidal, the main result in Li's paper (\cite{Li} main theorem
page $44$) assures the existence of at least a matrix $a_0$ of maximal rank $n$ 
whose corresponding Fourier coefficient  $\xi_{a_0}(g)$ is non zero. However, this contradicts proposition \ref{sonmiodhl1}, since all the forms appearing in the decomposition must be degenerate. Now proposition \ref{howelemma} affirms that if a representation is degenerate for one place, then it must be degenerate for all places. Consequently, any global discrete parameter whose component at the  place $p$ is the local representation $Ind_{SO_6\times \G_m N}^{SO_8}(\epsilon \otimes {\bf
1}\otimes \chi)$,  has to be residual.
\end{proof}

In the next subsection we will extend these results to the non-unitary case.

\subsection{Extension to the non-unitary case}
In order to formulate the extension to the non-unitary case, we quote a result from Dixmier:
\begin{teo}
Let $G$ be a group of type $I$ and $\rho$, a representation of $G$  into a Hilbert space $H$.
\begin{enumerate}
\item  There exists a Borel  function $\pi \to m(\pi)$ from $\widehat{G}$ to $\{1,2,\dots, \infty\}$ and a positive measure $\mu$ on $\widehat{G}$ such that
\begin{equation}\label{dixi1}
(\rho,H)= \int_{\widehat{G}}\widehat{\oplus}m(\pi)H_\pi d\mu(\pi)
\end{equation}
\item This expression is essentially unique. This means that if $\rho$ possesses two expressions \eqref{dixi1} for measures $\mu$ and $\mu'$ and functions $m$ and $m'$, then $\mu$ and $\mu'$ are equivalent and the functions $m$ and $m'$ are equal almost everywhere.
\end{enumerate}
\end{teo}
The utility of \eqref{dixi1} is clear, the Borel function $m$ indicates the multiplicity of the representation $\pi$ in the integral decomposition. Dixmier's formulation says
\begin{equation}\label{dixi2}
(\rho,H)=\bigoplus_{n=1}^\infty  \int_{\widehat{G}}\widehat{\oplus}H_\pi d\mu_n(\pi)
\end{equation}
where the measures $\mu_n$ are mutually disjoint.

\begin{defn}
Let $\rho$ be a representation of a group $G$ of type $I$, having a decomposition of type \eqref{dixi1} or \eqref{dixi2}. The support of the measure $\mu$, that is a closed subset of $\dualg$, is called the \emph{support} of $\rho$ and is denoted by $Supp(\rho)$. If $\pi \in \dualg$ we say that $\pi$ is \emph{weakly contained} in $\rho$ if $\pi \in Supp(\rho)$.
\end{defn} 

Let $k$ be a $p$-adic field. We shall consider the group $G$ defined over the $p$-adic field $k$. This property can then be characterized in terms of the matrix coefficients.

\begin{prop}\label{equisup}
Let $G$ be a group of type $I$, $\rho$ an unitary representation  of $G$ and $\rho_1 \in \dualg$. The following statements are equivalent:
\begin{enumerate}
\item The representation $\rho_1$ is weakly contained in $\rho$.
\item Every matrix coefficient $c_v$ of $\rho_1$ is a limit, uniformly on every compact subset of $G$, of a sequence of positive linear combinations of coefficients of $\rho$.
\end{enumerate}
\end{prop}
For the proof see \cite{Clo} or \cite{Dix}.

We state the main result of this subsection.
\begin{teo}\label{unitext}
Let $N$ be a unipotent abelian group and suppose we have an  increasing exhaustive sequence $\{N_l\}$ of compact subgroups of $N$. If $\chi$ is a character of $N$ weakly contained in $\rho$ for every $l$, then the restriction $\chi |_{N_l}$ is weakly contained in $\rho$. Conversely, if for every $l$ the restriction $\chi |_{N_l}$ is weakly contained in $\rho$, then $\chi$ is weakly contained in $\rho$.
\end{teo}
\begin{proof}
We can suppose that $\{N_l\}$ is the sequence $\exp(\fp^{-l}\fnn(\calo))$, with $\calo$ the ring of integers of $k$, $N \simeq (\G_a)^r$ defined over $\Z$, $\fnn \simeq k^r$ is the Lie algebra of $N$ and we set $\fnn(\calo)= \calo^r$. We know that $\rho$ admits a decomposition of type \eqref{dixi2}, so 
\[
H= \bigoplus_1^\infty L^2(\dualn, d\mu_n)
\]
If $\chi \in Supp(\rho)$, this part is an immediate consequence of proposition \ref{equisup} since if every coefficient $c_v$ of $\rho$ is a uniform limit on every compact subset of $G$, in particular, it is, over each $N_l$.
We know that $N$ can be canonically identified with its Lie algebra, i.e. $N \simeq \fnn$. In particular, if we chose a basis of $\fnn$, such that $\fnn(\calo)$ is well defined, we have $\fnn(\calo) \simeq N(\calo)$. Thus, their dual groups are also isomorphic $\dualn \simeq \fnn^*$. This remark is meaningful, since we can consider $\chi$ as an element of $\fnn^*$. The compact subgroups form an increasing sequence and hence for $i>j$ we have $N_i \supset N_j$.  These compact subgroups can be written as
\[
N_l = \exp(\fp^{-l}\fnn(\calo)) \subset N
\]
So if $X\in \fnn$, $\exp(X)=x\in N$. Suppose we have fixed a non-trivial additive character $\theta$ of $k$,  whose conductor is $\calo$. Thus, the character $\chi(x)= \theta(\lambda(X))$, for some form $\lambda$.
We have seen that $H$ is decomposed as
\[
H= \int_{\widehat{N}}\widehat{\oplus}m(\chi)H_\pi d\mu(\chi)
\]
where the spaces $H(\chi)$ are invariant, irreducible and isomorphic to $\C$ since $N$ is abelian. If we consider $\chi|_{N_l}$ we have
\[
H= \int_{\widehat{N}_l}\widehat{\oplus}m(\chi_{N_l})H_\pi d\mu(\chi_{N_l})
\]
If $\chi|_{N_l} \in Supp(\mu)$ there exists $\phi \in Supp(\mu)$ such that $\phi = \chi + \lambda$ and $\lambda |_{\fp^{-l}\fnn(\calo)} \subset \calo$. Thus $\lambda |_{\fnn(\calo)} \subset \fp^l$. Hence $\lambda \in \fp^l \fnn(\calo)^*$, where 
\[
\fnn(\calo)^* = \{\lambda \in \fnn^* \mid \left<\lambda, \fnn(\calo)\right> \subset \calo\}
\] 
Thus if$\chi|_{N_l} \in Supp(\rho)$ for all $l$, then $\chi \in Supp(\rho)+\fp^l \fnn(\calo)^*$.   Since $\fp^l \fnn(\calo)^*$ is a base of neighbourhoods of zero and $Supp(\rho)$ is closed, then $\chi \in Supp(\rho)$
\end{proof}

\begin{cor}\label{cdxrunt}
The local induced representation $Ind_{SO_{n+2}\times \G^{n-1}_m N}^{SO_{2n}}(\epsilon\otimes \bf{1}\otimes \chi)$ considered in this section, when $\chi$ is no longer unitary, is singular.
\end{cor}
\begin{proof}
We consider only the non-tempered case.  We consider the induced representation 
$Ind_P^{\son}(\epsilon\otimes |\,|^s)$ where $s \in(-\frac{1}{2},\frac{1}{2})$, $s$ real number. The representation
$Ind_P^{\son}(\epsilon\otimes |\,|^s)$ is realized as a compact induced representation
 $Ind_{K\cap P}^K(\epsilon\otimes 1|_{K\cap P})$, where $K$ is a maximal compact subgroup. Since the subgroups $N_l$ are compact (and absorbed modulo conjugacy for every neighborhood of the identity), they embed, up to conjugacy, into $K$. The restriction of $Ind_P^G(\epsilon\otimes |\,|^s)$ to $N_l$ does not depend on $s$. The corollary follows from \ref{sonmiodhl1}.
\end{proof}

\subsection{More eccentric parameters}
We consider first the parameter given by $\psi= sp_{n+1}+\sum_{i=1}^{m-1}\chi_i$, where $\chi_i$ is a character of  $\G_m= \Q^\times$, for all $i$. The conditions $det=1$ and the self-duality imply that one of these characters must be 
trivial. In this parameter $\psi$,  $sp_{n+1}+1$ occurs as the trivial representation of  $SO_{n+2}$ and thus we are in the case of
\[
Ind_{SO_{n+2}\G_m^{m-1} N}^{\son} (\epsilon\otimes {\bf 1})\otimes  \chi
\]
which is singular.

Consider now parameters $\psi$ having  eigenvalues (in their Hecke matrix) bigger than those of $sp_n$. These parameters are of type
\begin{equation}\label{parmtexcnt}
\psi= \epsilon\otimes sp_k \oplus \sum r_i \otimes sp_i,
\end{equation}
having blocks $sp_k$ with $k> n+1$. We consider first, the case $k=n+3$, corresponding to the trivial representation of $SO_{n+4}\subset \son$. The idea now is to realize $sp_{n+3}$ as a singular subrepresentation, where $sp_{n+1}$ occurs. More precisely, we have
\[
\xymatrix{
sp_{n+3}  \ar@{^{(}->}[r] & Ind_{\G_m SO_{n+2} N}^{\son}(\chi,sp_{n+1})
} 
\]
Therefore, when we consider the whole parameter denoted by $\psi'$ we have
\[
\xymatrix{
\psi'= sp_{n+3} \oplus \sum sp_{i}\otimes r_i  \ar@{^{(}->}[r] & Ind_{GL_r \G_m
  SO_{n+2}N}^{\son}(\chi,sp_{n+1})\otimes\xi
}. 
\]
Here $\xi$ is a character of the torus of $GL_r$ and $\chi$ a character of $\G_m$. Since we have the initial  hypothesis that $sp_{n+1}$ is singular, the parameter  $\psi'= sp_{n+3} \oplus \sum
sp_{i}\otimes r_i $ is also singular. We generalize this argument for bigger blocks with an easy induction. This lead us to the following result. 
\begin{teo}\label{teosingexct}
Let $\psi$ be a discrete parameter of the split group $\son$ having the form
\[
\psi= sp_{k} \oplus \sum r_i\otimes sp_{i}
\]
with $k$ odd and $k>n+1$. Then $\psi$ is globally singular in Howe's sense. More precisely, the representations of  $SO_{2n}(\Q_v)$ associated  to the parameter $\psi$, are singular and cannot appear in the space of cuspidal forms.
\end{teo}
\begin{obs}
Note that the  conclusion established  in this theorem, does not depend on Arthur's conjectures. What does depend on Arthur's description, is the fact that every representation appearing in the  discrete spectrum is associated to a discrete parameter $\psi$.
\end{obs}
\section{Hecke operators}

\subsection{Bounds for symplectic groups}
Duke, Howe and Li in their paper \cite{DHL} estimate the eigenvalues of Hecke operators for Siegel modular forms. In their paper, they use all the power of the theory of low rank representations developed by Howe \cite{How} and the non-existence of singular cuspidal forms, a result due to Li (\cite{Li}). They consider the Hecke operator associated to the double coset in $GSp_{2n}(\Q_p)$ given by
\[
T_p= K \begin{pmatrix} \underbrace{\begin{matrix} p & & \\ & \ddots & \\ & & p \end{matrix}}_n & \begin{matrix} & & \\ & & \\ & &  \end{matrix}\\ \begin{matrix}  &  &   \\  &  &  \\  &  & \, \end{matrix} & \underbrace{\begin{matrix} 1 & & \\ & \ddots & \\ & & 1 \end{matrix}}_n \end{pmatrix} K,
\]
where $K= GSp_{2n}(\Z_p)$. Their bounds are listed in the corollary $5.3$ and $5.5$ of \cite{DHL}.

\subsection{Bounds for $\son$}
We recall that for the case of cuspidal representations of $GL_n$, we have several approximations to Ramanujan's conjecture:
\begin{enumerate}
\item $|t^{\pm 1}| < p^{\frac{1}{2}}$  due to Jacquet and Shalika, see \cite{J-S}.
\item $|t^{\pm 1}|< p^{\frac{1}{2} - \frac{1}{n^2+1}}$  due to Luo, Rudnick and Sarnak.
\end{enumerate}
We recall also that we have natural maximal compact subgroups, e.g. $\son(\Z_p):=K_p$ for the finite places. We set $K_f=\prod_p K_p$ where the product runs over all the finite places. Fix a maximal compact subgroup $K_\infty$ of $\son(\R)$. We note $K=K_\infty\times K_f\subset \son(\A)$.
\begin{defn}
We define $H_p:=H(\son(\Q_p), K_p)$ to be the algebra of functions with compact support on
$\son(\Q_p)$, bi-invariant under $K_p$ . Any function $f$ in $H_p$ is constant in the double cosets $K_p x K_p$. Since $f$ has compact support, $f$ is a finite linear combination of characteristic functions of double cosets $ch_{K_p x K_p}$. Thus, these characteristic functions give a $\Z$-base for $H_p$. We shall call this algebra, the \emph{local Hecke algebra}. We set
\[
H_f(\son, K_f)=\bigotimes_p H(\son(\Q_p), K_p)
\]
\end{defn}

We have an action of convolution of $H_f(\son, K_f)$ on the space of functions on $\son(\Q)\backslash\son(\A) /K_f$. We shall describe the action of the Hecke operators on the $K_f$-invariant vectors of a representation appearing in the spectral decomposition of $\son$. We know that any  such automorphic representation $\pi$ decomposes into a tensor product $\pi=\otimes_p \pi_p$. For almost all $p$, $\pi_p$ is an unramified representation of $\son(\Q_p)$. In that case, we know that $\pi_p$ is induced from an unramified character $\chi$ of the maximal torus $T_0(\Q_p)$. Recall that the Weyl group can be described as the quotient $W= N(T_0)/T_0$, where $ N(T_0)$ is the normalizer of $T_0$ in $\son$. Thus, this character $\chi$ is defined up to the action of $W$. We denote by $X^*(T_0)$, the group of rational characters of $T_0$ and by $X_*(T_0)$, its group of cocharacters. Recall that in our case the Langlands group $^L\son$ is $\son(\C)\times \Gamma_\Q$,  with $\Gamma_\Q$ the Galois group of $\Q$. Fix a maximal torus $\widehat{T}_0$ of $\son(\C)$ and an isomorphism
\[
X^*(\widehat{T}_0)\simeq X_*(T_0)
\]
We denote by $t_1,\dots, t_n$, a basis of $X^*(\widehat{T}_0)$. The Satake transform induces an isomorphism between $H_p=H(\son(\Q_p), K_p)\otimes \C$, the local Hecke algebra and $\C[X^*(\widehat{T}_0)]^W= \C[t_i^{\epsilon_i}]$, the algebra of polynomials that are invariant under the Weyl group (here $\epsilon_i=\pm 1$). For any function $f\in H_p=H(\son(\Q_p), K_p)$ we denote its Satake transform by $Sf$.

For the calculation of the Satake transform, we shall appeal to minimal weights. Let $\Sigma^+$ be the set of positive roots  of $\son(\Q_p)$ and $\rho$ the half-sum of the positives roots. Let $P^+$ be the positive Weyl chamber:
\[
P^+=\{ \lambda \in X_*(T_0) \mid \left<\lambda, \alpha\right> \ge 0\quad
\text{for $\alpha \in \Sigma^+$}\}.
\]

This gives a decomposition of $\son(\Q_p)$
\[
\son(\Q_p)= \bigcup_{\lambda \in P^+} K_p \lambda(p) K_p 
\]
The characteristic functions $ch_\lambda$ of the double cosets $K_p
\lambda(p) K_p$ form a basis of $H(\son(\Q_p), K_p)$. On the other hand, the elements $\lambda \in P^+$ (which can be viewed as elements of $X^*(\widehat{T}_0)$) parametrize the irreducible finite-dimensional representations $V_\lambda$ of $^L\son$, with $\lambda$ the highest weight of $V_\lambda$. For $\lambda$ a minimal weight of ${}^L\son$, the Satake transform has  the form
\begin{equation}\label{ophck02}
Sch_\lambda= p^{\left<\lambda,\rho\right>}Tr(V_\lambda)
\end{equation}
This a standard result in Shimura varieties due to Kotwittz (\cite{Kot}, theorem $2.1.3$). These Satake transforms have been calculated before in several cases (see \cite{DHL} and \cite{C-U}).

We calculate this transform in our case. Recall that the maximal torus for $\son$ has the form
$\widehat{T}_0$
\[
\widehat{T}_0= \left\{ \begin{pmatrix} t_1 & & & & &\\ & \ddots & & &
  &\\ & & t_n & & &\\ & & & t_n^{-1} & &\\ & & & & \ddots & \\ & & & &
  & t_1^{-1}\end{pmatrix} \mid t_i\in \C^\times \right\}
\]

Thus a character $\chi\in \widehat{T}_0$ is associated to $(t_1,\dots, t_n)$ by
\begin{equation}\label{ophck2.5}
t_i=\chi\begin{pmatrix} 1 & & & & & & & &\\ & \ddots & & & & & & &\\  & & p
& & & & & &\\ & & & \ddots & & & & &\\ & & & & 1 & & & &\\ & & & & & \ddots  &
& &\\ & & & & & & p^{-1} & & \\ & & & & & & & \ddots & \\ & & & & & & & &
1\end{pmatrix} 
\end{equation}
where $p$ appears in the $i$-th place.

On the other hand, the Hecke matrix of the trivial representation is the "largest" representation of $SL_2(\C)$ embedded in the orthogonal group $\son$. Since $n$ is even, this representation has degree $2n-1$. The Hecke matrix of the trivial representation is given by
\begin{equation}\label{ophck03}
t_{p,Id}=  \begin{pmatrix} p^{n-1}& & & & & & & &\\ & p^{n-2} & & & &
  & & &\\ & & \ddots & & & & & &\\ & & & p & & & & &\\ & & & & 1 & & &
  &\\ & & & &
  & 1 & & &\\  & & & & & & p^{-1}&  &\\ & & & & & & & \ddots &\\  & & & & &
  & & & p^{1-n}\end{pmatrix} 
\end{equation}
This Hecke matrix allows us to calculate the Satake transform for the characteristic functions. Consider the simplest Hecke operator,  $T_p$, defined by the minimal weight $(1, 0, \dots, 0)$:
\begin{equation}\label{ophck04}
T_p = K_p \begin{pmatrix} p & & & & \\ & 1 & & &\\  & &\ddots  & &\\ &
  & & 1 &\\ & & & & p^{-1}\end{pmatrix} K_p 
\end{equation}
The Satake transform of the characteristic function $f_p$ associated to the double coset \eqref{ophck04} can be calculated by  \eqref{ophck02}. This calculation gives
\[
Sf_p= p^a(t_1+ \cdots + t_n+ t^{-1}_n + \cdots + t^{-1}_1)
\]
where the exponent $a$ is calculated as
\[
a=\frac{1}{2}\left( \sum_{i<j}X_i-X_j + \sum_{i<j}X_i+X_j\right)
\]
evaluated in the vector $(1,0,\dots, 0)$. One obtains $a= n-1$. In particular, we have this result.
\begin{lema}\label{sataketransf}
The degree  of the characteristic function  $f_p$ of the double coset \eqref{ophck04} is given by
\begin{equation}\label{eqsattransf}
Sf_p=\widehat{T}(triv)= p^{n-1}\left(\sum_{i=0}^{n-1}(p^i +
p^{-i})  \right)
\end{equation}
\end{lema}

Remark that according to the singularity argument given in the previous section, the cuspidal representations cannot have blocks $sp_k$ for $k \ge n$ in the case of $\son$ and $k \ge 4$ in the case of $SO_8$. This means that, if Arthur's conjectures are true, the parameters present in the cuspidal spectrum can only have blocks $sp_k$ with $k <n$. Once we consider the Hecke matrices of these cuspidal parameters we have:
\begin{teo}\label{tophkso8}
Let $\pi$ be a unitary representation of  $\so(\A_\Q)$, occurring  in the cuspidal spectrum of $\call^2(\so(\Q)\backslash \so(\A_\Q))$, unramified at $p$. We suppose that  Arthur's conjectures hold. If
\[
\pi(f_p) e_p= \lambda_p e_p
\]
where $e_p$ is the unramified vector of $\pi_p$, then we have the estimate
\begin{equation}\label{estso8}
|\lambda_p| \le p^{3}(p^{3/2}+\cdots + p^{-3/2})(\tau + \tau^{-1})
\end{equation}
where $(t, t^{-1})$ is the Hecke matrix of a  cuspidal representation of  $GL_2$ unramified at $p$,  $\tau= |t|$ and according to Kim and Sarnak, we have
\[
|\tau| \le p^{\frac{7}{64}}
\]
\end{teo}
We have the analogous result for the case of $\son$.
\begin{teo}\label{tophkso2n}
Let $\pi$ be  a unitary representation of $\son(\A_\Q)$, occurring in the cuspidal spectrum of 
$\call^2(\so(\Q)\backslash \son(\A_\Q))$ and unramified at $p$. We suppose that  Arthur's conjectures hold. If we have
\[
\pi(f_p) e_p= \lambda_p e_p
\]
with  $e_p$ the unramified vector of $\pi_p$, then we have the estimate
\begin{equation}\label{estso2nsnada}
|\lambda_p| = O\left(p^{n-1+ \frac{n-1}{2}-\epsilon(n)}\right),
\end{equation}
where $\epsilon(n)= \frac{1}{n^2+1}$ with the known estimates dues to Luo, Rudnick and Sarnak. If
 Ramanujan's conjecture holds, we have the estimate
 \begin{equation}\label{estso2nram}
|\lambda_p| = O\left(p^{n-1+ \frac{n-2}{2}}\right)
\end{equation}
\end{teo}
In the next sections we shall show that, with  Arthur's results, these estimates can be ameliorated for the parabolic forms.


\section{Local results}
In this section we compare the trace formula for the group $GL_{2n}$ and for the split form of the orthogonal group $SO_{2n}$. The group $\son$ is an endoscopic maximal subgroup of $\gnn \rtimes \theta$, the non-neutral component of the non-connected subgroup $\gnn\rtimes (\Z/2\Z)$. In this section $k$ denotes a local field.

\subsection{Calculation of characters}
In the subsection \ref{artasctele} we shall recall the notion of associated elements, that we will use freely in this subsection.

 It will be more useful to work with the split form of $\son$ given by the matrix $J^*_0$,
\[
J^*_0=\begin{pmatrix} & 1_n\\ 1_n & \end{pmatrix} 
\]

In the case of the group $\gnn$, we consider the Levi subgroup $M := GL_n$, that imbeds naturally into $\son$ as
\[
g\in GL_n\quad \mapsto \begin{pmatrix} g & \\ & ^tg^{-1}\end{pmatrix} \in \son
\]
The representation we are going to consider is the induced representation
\begin{equation}\label{art6eq1}
Ind_{GL_n N}^{\son}(|det|^s):= \pi
\end{equation}
where $N$ is the unipotent subgroup of the parabolic subgroup containing $M$, and $s\in \C$. From the side of $\gnn$, we consider the Levi subgroup $L:= GL_n \times GL_n$ and the representation
\begin{equation}\label{6eq1.1}
Ind_{GL_n\times GL_n N^\prime}^{\gnn}(|det(m_1)|^s |det(m_2)|^{-s}):= \Pi
\end{equation}
where $N'$ is the corresponding unipotent subgroup.\footnote{Please note that in this section, we are using the notations $\pi$ and $\Pi$ for representations that are different from those of past chapters.}

\begin{obs}
We know (see \cite{Cl1}) that the character $\Theta_\pi(g)=0$, except if $g$ is conjugate to an element of $M$. Also, since it is enough to know $\Theta_\pi(g)$ almost everywhere, we shall calculate this character only on the regular elements.
\end{obs}

Let $T_0$ be the maximal torus of the center of $M$, so $T_0$ is the split component of $M$ denoted by $A_M$. So $A_M=T_0= Diag(t, \dots, t, t^{-1}, \dots, t^{-1})$. On the other hand, if $S$ is a maximal torus of $M= GL_n$, we can associate to it a partition of $n=  n_1+\cdots + n_r$, such that each $n_i$ corresponds to an extension of degree $n_i$ of a field $K_{n_i}/k$. Thus, $K_{n_i}^\times \subset GL_{n_i}(k)$ and
\[
S=  K_{n_1}^\times \times \cdots \times K_{n_r}^\times
\]
For the calculation of the character we follow an article of van Dijk \cite{vDj}. In van Dijk's formulation, the character of $\pi$ on a regular element $\gamma$ is
\begin{equation}\label{art6eq2}
\Theta_\pi(\gamma)=\sum_{j\in W(A, A_S)}\theta_{j\rho}(\gamma)\frac{|D_M^{j}(\gamma)|^{\frac{1}{2}}}{|D_{\son}(\gamma)|^{\frac{1}{2}}},
\end{equation}
where $W(A,A_S)$ is the set of embeddings $j:A \to A_S$ that are $\son$-realized, i.e., they come from an inner automorphism. The set $W(A,A_S)$ is in general difficult to describe. The term $D_{\son}(\gamma)$ is defined in the following way. Let $t$ be an indeterminate and $l$ the rank of $\son$. So $D_{\son}(\gamma):= D(x)$
\[
det(t+1-Ad(x))= D(x)t^l + \text{terms of higher degree}
\]
Also, for the regular elements $x\in \son$, we have $D(x)\neq 0$.  We can calculate this in terms of roots.
\[
det(t-(Ad(x)-1))= t^l\prod_\alpha (t -(x^\alpha -1))
\]
where the product is over all the roots $\alpha$. Therefore, we have
\[
D_{\son}(x)= \prod_\alpha (x^\alpha-1)= det(Ad(x)-1)|_{\mathfrak{so}/\mathfrak{t}}
\]
We have denoted by $\mathfrak{so}$ and $\mathfrak{t}$, the Lie algebras of $\son$ and $T_0$, respectively. When we make this same calculation for the Levi subgroup $M$, we obtain 
\[
D_M(x)= det(Ad(x)-1)|_{\mathfrak{m}/\mathfrak{t}}
\]
Therefore,
\[
D_{M}(x)= det (Ad(x)-1)|_{\mathfrak{m}/\mathfrak{t}}= \prod_{\substack{\alpha >0,\\ \alpha \notin M}} (x^\alpha -1)\prod_{\substack{\alpha <0, \\ \alpha \notin M}} (x^\alpha -1)= \prod_{\substack{\alpha >0,\\ \alpha \notin M}} (x^\alpha -1)(x^{-\alpha}-1)
\]
This term is stably invariant since it is algebraic. Thus, it is enough to calculate it for an element $x \in T(\overline{F})= \{ x_1, \dots, x_n, x_1^{-1}, \dots, x_n^{-1} \}$.

A straightforward calculation shows, that this term in the twisted case is
\begin{equation}\label{art6eqtord00}
D^\theta_{L}(x) = \prod_{i\neq j}(x_i x_j -1)(x_i^{-1}x_j^{-1}-1)\prod_{i=j}(x_i+1)(x_i^{-1}+1) 
\end{equation}
This implies that in the case of an elliptic torus, i.e., when there is only one field $K_n$, the set $W(A, A_S)$ acts on the torus $A_S$ in only two ways:
\[
W(A, A_S) = \begin{cases} (x, x^{-1})   \to(x, x^{-1})\\(x, x^{-1}) \to (x^{-1}, x)  \end{cases}
\]
However, it is not necessary to give an explicit description of this set for every possible torus $S \subset M$. Rather, we need to understand how this set is related to the one given in the twisted case. So, the problem is how to associate to a torus $S \subset M$, a torus $T$ contained in $L= GL_n\times GL_n$. This problem is very closely related to the construction of the \emph{norm} for the elements of $\son$. The abstract construction of the norm is due to Kottwitz and Shelstad \cite{K-S} and in certain cases it has been described explicitly by Waldspurger (see \cite{Wal}).

We are going to recall briefly the construction of the norm associating the (stable) conjugacy classes in $\son$ to the (stable) twisted conjugacy classes in $\gnn$. In the next subsection we shall construct the norm in our case, which is simpler since conjugacy and stable conjugacy are the same for the regular elements of $M$. Thus, if we consider a torus $S$, the norm $\cala$ of $S$  can be described by the $\theta$-torus contained in $L= GL_n\times GL_n$:
\[
T= \{ (x,1) \mid x  \in S\} = S\times 1
\]
Note that we need to calculate also the twisted centralizer $T'$ (which is also a torus) of $T$ in $L$. So, let $m= (m_1, m_2)\in L$ be an  element in the $\theta$-centralizer of $T$. We will use the notation $m^\theta$ for $\theta(m)$. Let $t\in T$, then
\[
\theta(m)t m^{-1}= m^\theta t m^{-1}= (^tm^{-1}_2, ^tm^{-1}_1)(x,1)(m_1^{-1},m_2^{-1})= t= (x,1)
\]
This implies that $^tm^{-1}_1=  m_2$ on one hand, and on the other, $m_1 x m^{-1}_1= x$. From this we deduce that $m_1 \in S$ and that it commutes with $x\in S$. Hence, the $\theta$-centralizer $T'$ is
\[
T'= \{ (t, ^t t^{-1})\mid t \in S \}.
\]
Since $T'$ has dimension $n$, it follows that it is the $\theta$-centralizer in $\gnn$. Next, we describe the set of embeddings $\widetilde{j}:A_{L}^\theta \to A_{T^\prime}$, where the embedding $\widetilde{j}$ is realized in $\gnn\rtimes \theta$. Thus, $\widetilde{j}$ has the form $\widetilde{j}(a) = (Ad(g)\circ \theta) (a)$, $g\in \gnn$.

In general, let $T'$ be the $\theta$-centralizer of a torus $T$ that is the image under the norm of a torus $S \subset M$. Next, we compare the embeddings $j: A_M \to A_T$ and $\widetilde{j}:A_{L}^\theta \to A_{T^\prime}$. We have, thus, a diagram like this
\[
\xymatrix{ A_M \ar[r]^{j} \ar[d]^{\simeq} & A_T \ar[d]^{\simeq}\\ A_{L}^\theta \ar[r]^{\widetilde{j}} & A_{T^\prime}
}
\]
\begin{lema}\label{art6lema}
If $j$ is $\son$-realized, then  $\widetilde{j}$ (determined by the diagram) is $\gnn,\theta$-realized.
\end{lema}
\begin{proof}
It is clear that $A_M \simeq A_L^\theta$ and $A_T \simeq A_{T^\prime}$. Let $j(a)=gag^{-1}$, $g\in \son$ and $a\in A_M$. We have $\theta(a)= a^{-1}$. Thus, if $x \in \gnn$
\[
x\theta(a)x^{-1}= xa^{-1}x^{-1}
\]
but the morphism $a \to a^{-1}$ is realized by conjugacy in $\gnn$. Therefore, if we consider any morphism $j$ as the associated morphism $\widetilde{j}$, then $\widetilde{j}$ is $\gnn$-realized. Let $\widetilde{j} = gag^{-1}$ be now an embedding $A_L^\theta \to A_{T^\prime}\simeq A_T$. We know that $A_T$ is a product of tori
\[
A_T= (\overbrace{\Q_p^\times \times \cdots \times \Q_p^\times}^{\text{$r$ times}})\times (\overbrace{\Q_p^\times \times \cdots \times \Q_p^\times}^{\text{$r$ times}})
\]
\[
(x, x^{-1}) \mapsto (x_1, \dots, x_r, x^{-1}_1, \dots, x^{-1}_r)
\] 
Therefore, for each coordinate, we have either $x\mapsto x$ or $x  \mapsto x^{-1}$. We can represent this by
\[
(x, x^{-1}) \mapsto (x^{\epsilon_1}, \dots, x^{\epsilon_r},  x^{-\epsilon_1}, \dots,  x^{-\epsilon_r})
\]
We have, thus, two possibilities.
\begin{enumerate}
\item  The morphism $(x, x^{-1}) \mapsto (x^{\epsilon_1}, \dots, x^{\epsilon_r},  x^{-\epsilon_1}, \dots,  x^{-\epsilon_r})$ is such that the product of morphisms 
\[
\epsilon_1^{n_1}\cdots \epsilon_r^{n_r}= 1
\]
\item This morphism is not realized in $\son$. Therefore we have
\[
\epsilon_1^{n_1}\cdots \epsilon_r^{n_r}= -1,
\]
hence, the morphism has the form $x \to x^{-1}$. We denote it by $\underline{\epsilon}$. Without loss of  generality, we can assume $\epsilon_1=\underline{\epsilon}$. Therefore, the first block $n_1$ has to be odd. Then the morphism $j$ has the form
\[
j(x,x^{-1})= (\overbrace{x^{-1}, \dots , x^{-1}}^{n_1}, x^{\epsilon_2}, \dots, x^{\epsilon_r}, \overbrace{x^{}, \dots, x^{}}^{n_1}, x^{-\epsilon_2}, \dots,  x^{-\epsilon_r} ).
\] 
\end{enumerate}
Let $w$ be the matrix exchanging the two $n_1$ blocks (in the first and second $n$ coordinates). Thus, after composing with $w$, we obtain 
\[
w\circ j(x, x^{-1})= (\overbrace{x, \dots , x}^{n_1},  x^{\epsilon_2}, \dots, x^{\epsilon_r}, \overbrace{x^{-1}, \dots, x^{-1}}^{n_1} x^{-\epsilon_2}, \dots,  x^{-\epsilon_r} )
\]
Therefore, $w\circ \widetilde{j}$ is realized in $\son$.
\end{proof}

In conclusion, we have characterized completely the set $\{\widetilde{j}\}$.

\begin{lema}\label{setsj}
 The set of the morphisms $\widetilde{j}$ has the form
 \[
\{ \widetilde{j}\}= \{j\} \coprod \{w\circ j\}
\]
\end{lema}

Next, we calculate the characters for both groups with the given data. For $\son$, we have
\begin{equation}\label{art6eq5}
\Theta_{\son, \pi}(t)= \sum_{j:A_M \to A_T} \theta_{j\delta}(t)\frac{|D_{M^j}(t)|^{1/2}}{|D_{\son}(t)|^{1/2}}.
\end{equation}
We know that  $j(A_M)\subset A_T$ and thus $T$ commutes with  $j(A_M)$, which implies that  $T\subset j(M)$ and $j(M)= gMg^{-1}$ if $g$ is a representative of the embedding $j$.

The term $\delta$ corresponds to the inducing representation $|det(m)|^s$. Then $\theta_{j \delta}(t)= \delta(g^{-1}tg)$. Furthermore, 
\[
D_{j(M)}(t)= D_M(g^{-1}t g),
\]
and the terms  $|D_{M^j}(t)|^{1/2}$  are invariant by conjugacy.

Next, we consider the twisted analogue of the formula for characters. We consider now the representation
\[
\Pi= Ind^{GL_{2n}}_{GL_n\times GL_n \times N_{2n}} |det(m_1)|^s
\]
We need to define an intertwining operator $A_\theta$ between $\Pi$ and $\Pi\circ \theta$. Note that the parabolic subgroup $P= GL_n \times GL_n \times N'$ is $\theta$-stable. On the space $\C$ of the inducing representation $|det(m_1)|^s$, which will be denoted by $\delta_{2n}$, we define simply $A_\theta=1$. This  intertwines $\delta_{2n}$ and $\delta_{2n}\circ \theta$. Thus,
\[
A_\theta f(g) = f (\theta g)
\]
We shall note $\Theta_{\gnn, \theta}^{\Pi} (g)$ the twisted character of $\Pi(g)A_\theta$ for a regular element $g\in \gnn$, (that means that the conjugacy class associated to it through the norm is regular in $\son$). It is a locally constant function on the set of regular elements (see for instance \cite{Cl1}).

The character $\Theta_{\gnn, \theta}^{\Pi} (g)$ has, as support, the set of regular elements of $\gnn$ that are  $\theta$-conjugate to $M$ (\cite{Cl1}). Following line by line van Dijk (\cite{vDj}), we obtain now, in the twisted case:
\begin{equation}\label{art6eqvdj0}
\Theta_{\gnn, \theta}^{\Pi}(t)= \sum_{\widetilde{j}}\Theta_{\widetilde{j}\delta_{2n},\theta}(t)\frac{|D_{L^{\widetilde{j}},\theta}(t)|^{1/2}}{|D_{\gnn,\theta}(t)|^{1/2}}
\end{equation}
In this formula, $\Theta_{\widetilde{j}\delta_{2n},\theta}(t)$ is the twisted character (equal to the character since $A_\theta= 1$) of $\delta_{an}$, transported by $\widetilde{j}$ (that maps the $\theta$-torus containing $t$ into $M$) and evaluated at $t$. The term $D_{M^{\widetilde{j}},\theta}(t)$ corresponds to the term $D_{M,\theta}(t)$  transported by $\widetilde{j}$ and evaluated at $t$.  We recall that for $\son$, the term $D_{\son}(t)$ is 
\[
det(Ad(t)-1)|_{\mathfrak{so}/\mathfrak{t}}
\]
Here, let $\mathfrak{t}_1$ be centralizer of $t$ in $\mathfrak{gl}_{2n}$, i.e.,
\[
\mathfrak{t}_1= \{ X \in \mathfrak{gl}_{2n}\mid (Ad(t)\circ \theta)X=X\}
\]
Thus, 
\[
D_{\gnn,\theta}(t)= det(Ad(t)\circ \theta -1)|_{\mathfrak{gl}_{2n}/\mathfrak{t}_1}
\]
Note that the discriminants $D_{\gnn,\theta}(t)$ and $D_{M,\theta}(t)$ are algebraically invariant under $\theta$-conjugation. On the other hand, the centralizer of $\mathfrak{t}'= Lie(\gnn^{\theta t})$ in $\mathfrak{gl}$ is a torus $\mathfrak{t}$ of dimension $2n$. In the description given before lemma \ref{art6lema}, this is the set $\{ (t,\, {}^ts^{-1}) \mid t,s \in S\}$.

The automorphism $Ad\circ \theta$ acts by $1$ on $\mathfrak{t}'$ and by $-1$ on a complementary space in $\mathfrak{t}$, thus
\begin{equation}\label{art6eqatn0}
D_{\gnn,\theta}(t)= 2^n det(Ad(t)\circ \theta -1)|_{\mathfrak{gl}_{2n}/\mathfrak{t}}:= 2^n D_{\gnn,\theta}^{1}(t)
\end{equation}
We obtain in an analogous way,
\begin{equation}\label{art6eqatn1}
D_{L,\theta}(t)= 2^n D_{L,\theta}^{1}(t)
\end{equation}

On the other hand, the action of $\theta$ on $M$ can be identified (through a variable change) with the automorphism
\[
(m_1,m_2)\mapsto (m_2, m_1)
\]
that is, the Galois automorphism in the split case $k\times k/k$ (\cite{A-C}, \S $I.5$). If $t=(x,1)$ and $t'= (x, ^tx^{-1})$ are associated, a straightforward calculation gives
\begin{equation}\label{art6eqatn2}
D_{L,\theta}^{1}(t)= D_{M}(t')
\end{equation}
Thus, if we rewrite the calculation of the twisted character, one obtains
\begin{equation}\label{6eqatn3}
\Theta_{\gnn,\theta}(t)= \sum_{\widetilde{j}}\Theta_{\widetilde{j}\delta_{2n}}(t)\frac{|D_{L^{\widetilde{j}},\theta}^{1}(t)|^{1/2}}{|D_{\gnn,\theta}^{1}(t)|^{1/2}}
\end{equation}
We recall that, for $t=(x,1)$ and $t'= (x, ^tx^{-1})$ associated, Kottwitz and Shelstad (\cite{K-S}, page 46) define the factor 
$\Delta_{IV}(t,t')$
\[
\Delta_{IV}(t,t')= \frac{|det(Ad(t)\circ \theta-1)|_{Lie(GL_{2n})/Lie(Cent(GL_{2n}^{\delta \theta}))}^{1/2}}{|det(Ad(t')-1)|_{Lie(\son)/Lie(T_{\son})}^{1/2}}= \frac{|D_{\gnn,\theta}^{1}(t)|^{1/2}}{|D_{\son}^{1}(t')|^{1/2}}
\]
Let $\alpha$ be the outer automorphism of $\son$ given by $g\mapsto wgw^{-1}$  ($w$ is a representative of $Out(\son)= O_{2n}/\son$) and $\pi'= \pi \circ \alpha$.

\begin{obs}\label{obsmatrixhecke}
The representation $\pi$ is not isomorphic to the representation $\pi'$. Indeed, the Hecke matrix for $\pi$ at the unramified places is the diagonal matrix
\[
t_{\pi,p}= diag(p^{\frac{n-1}{2}}t, p^{\frac{n-3}{2}}t, \dots p^{\frac{1}{2}}t, p^{-\frac{1}{2}}t^{-1}, \dots, p^{-\frac{n-1}{2}}t^{-1})
\]
On the other hand, the Hecke matrix of $\pi'$ is 
\[
t_{\pi^\prime,p}= D(p^{\frac{n-1}{2}}t, p^{\frac{n-3}{2}}t, \dots p^{-\frac{1}{2}}t^{-1}, p^{\frac{1}{2}}t^{}, \dots, p^{-\frac{n-1}{2}}t^{-1})
\]
The automorphism $\alpha$ is realized in $\widehat{SO}_{2n}$ by the matrix
\[
\begin{pmatrix} 1_{n-1} & & &\\ & & 1 &\\ & 1 & &\\ & & & 1_{n-1}\end{pmatrix}
\]
Since $|t| < p^{\frac{1}{2}}$, these two matrices are not conjugate under the Weyl group of $\son$.

\end{obs}
We can now state the main theorem of this section.
\begin{teo}\label{arteo6chign1}
For associated $t\in \gnn$,  $\theta$-regular  and $t'\in \son$ we have the identity
\begin{equation}\label{6eqtc1}
\Theta_{\gnn, \theta}(t)= (\Delta_{IV}(t,t'))^{-1}\left(\Theta_{\son, \pi}(t')+\Theta_{\son, \pi^\prime}(t')  \right)
\end{equation}
where  $\pi'$ is the image of $\pi$ under $\alpha$, the outer automorphism of $SO_{2n}$.
\end{teo}
\begin{proof}
Lemma \ref{art6lema} implies that we need only to compare each term associated to the morphism $j$, in the equations \eqref{art6eq5} and \eqref{6eqatn3}, applied to the element $t'$. We have $\Theta_{j\,\delta_{2n}}(t)= \theta_{j\delta}(t)$ and $\Theta_{(w\circ j)\delta_{2n}}(t)= \theta_{j\delta}(wt)$, which appear in $\Theta_\pi$ and $\Theta_{\pi^\prime}$. On the other hand, the discriminants are algebraically invariant. Thus, we have to calculate them only for $j=1$. The equation \eqref{art6eqatn2} and the formula for the factor $\Delta_{IV}$ give
\[
\frac{|D^1_{L,\theta}(t)|^{1/2}}{|D^1_{\gnn,\theta}(t)|^{1/2}}= (\Delta_{IV}(t,t'))^{-1}\frac{|D_{M}(t')|^{1/2}}{|D_{\son}(t')|^{1/2}}
\]
Thus, we have an explicit form for the factor $\Delta_{IV}$.
\end{proof}

Next, we establish the result for associated functions. 
\begin{prop}\label{art6prop71}
Let $\phi$ and $f$ be associated functions on $\gnn$ and $\son$, respectively. Then we have
\[
\Theta^{\Pi}_{\gnn,\theta}(\phi)= \frac{1}{2}|2^n|_F\left(\Theta_{\son,\pi}(f)+\Theta_{\son,\pi^\prime}(f) \right)
\]
\end{prop}
\begin{proof}
Let ${T}$ be a family of $\theta$-conjugacy classes of $\theta$-tori in $\gnn$ and let $T_G$ their image under the norm $\cala$.  We choose in each $T_G$, a set of regular elements $U_G$, such that $\coprod U_{\GO}$ is a set of representatives of the regular support of $\pi$ and $\pi'$, modulo $\son$-conjugation. Thus, for $\gamma$ and $\delta$ associated, $\delta= (x,1)$, $\gamma = (x, ^tx^{-1})$, we have by Weyl integration:
\[
(\Theta_\pi + \Theta_{\pi^\prime},f)=\sum_{T_{G}}\int_{U_{G}}|D_{\son}(\gamma)|(\Theta_\pi+\Theta_{\pi'})(\gamma)O_\gamma(f) d\gamma
\]
\[
=\sum_{T_{G}}\int_{U_{G}}|D_{\son}(\gamma)|\Delta_{IV}(\gamma,\delta)\Theta_{\Pi,\theta}(\delta)O_\gamma(f)d\gamma 
\]
\[
=\sum_{T_{G}}\int_{U_{G}}|D_{\son}(\gamma)|\Delta_{IV}(\gamma,\delta)^2\Theta_{\Pi,\theta}(\delta)TO_\delta(\phi) d\delta
\]
\[
=\sum_{T_{G}}\int_{U_{G}}|D^1_{\gnn,\theta}(\delta)|\Theta_{\Pi,\theta}(\delta)TO_\delta(\phi) d\delta
\]
\[
=|2^{-n}|_F\sum_{T_{G}}\int_{U_{G}}|D_{\gnn,\theta}(\delta)|\Theta_{\Pi,\theta}(\delta)TO_\delta(\phi) d\delta
\]
On the considered elements, conjugacy and stable conjugacy are the same, but the map norm   $\cala$ identifies two elements exactly on the set of $\theta$-regular elements. Thus, the last integral is
\[
2|2^{-n}|_F\sum_{T}\int_{V_{\gnn}}|D_{\son,\theta}(\delta)|\Theta_{\Pi,\theta}(\delta)TO_\delta(\phi) d\delta
\]
where $\{T\}$ and $\{V_{\gnn}\subset T\}$ form a set of representatives for the $\theta$-conjugation. Finally, the last integral is equal to
\[
\int_{\gnn}\Theta_{\Pi,\theta}(\delta)\phi(\delta)d\delta= \Theta_{\pi,\delta}(\phi).
\]
\end{proof}

\subsection{Consequences for unramified functions}
Let $\phi$ be a function  in the Hecke algebra $\calh_{2n}:= \calh(\gnn, \gnn(O_k))$ (where $O_k$ is the  maximal compact subgroup of $k$) and such that $f:= \lambda \phi$ is the image by the natural homomorphism
\[
\calh_{2n}\to \calh_G= \calh(\son, \son(O_F))
\]
given by the Satake transform $\widehat{\phi} \mapsto \widehat{f}$, where the element $\widehat{\phi}(t_1, \dots, t_{2n})$ belongs to the polynomial algebra $\C[t_1, t_1^{-1}, \dots, t_{2n}, t^{-1}_{2n}]^{\mathfrak{S}_{2n}}$ and  $\mathfrak{S}_{2n}$ is the symmetric group in $2n$ variables. The function  $\widehat{f}(t_1,\dots, t_n, t^{-1}_n, \dots, t^{-1}_1)\in \C[t_1, \dots, t_n]^W$, (where $W$ is the Weyl group of $\son$) is obtained by restriction of $\widehat{\phi}$. Suppose that we have measures $dg_{2n}$, $dg$ on the groups $\gnn(k)$ and $\son(k)$, respectively, such that
\[
1= \int_{\gnn(O_k)} dg_{2n}= \int_{\son(O_k)} dg
\]
 Let $t_\Pi$, $t_\pi$, $t_{\pi^\prime}$ the Hecke matrices of $\Pi$, $\pi$ and $\pi'$, respectively at the unramified places in the maximal tori $\widehat{T}_{2n}$ and  $\widehat{T}_G$. Thus, for every $\phi \in \calh_{2n}$
\[
trace\,\left(\Pi(\phi)A_\theta\right)= trace\,\Pi(\phi)= \widehat{\phi}(t_\Pi)
\]
Suppose that $\phi$ and $f=\lambda \phi$ are associated (in the sense of orbital integrals). Thus for $\phi \in \calh_{2n}$:
\[
trace\,(\pi(\lambda \phi))= \widehat{(\lambda \phi)}(t_\pi)= \widehat{\phi}(t_\pi)
\]
and also
\[
trace\,(\pi'\,(\lambda \phi))= \widehat{(\lambda \phi)}(t_{\pi^\prime})= \widehat{\phi}(t_{\pi^\prime})
\]
Since $t_\pi$ and $t_{\pi^\prime}$ are conjugate in $\gnn(\C)$ we find that
\begin{equation}\label{artegalitetraces}
trace\,\left(\Pi(\phi)A_\theta\right)= trace\,\pi(f)= trace\,\pi'(f)
\end{equation}
At least, if $p\neq 2,\infty$.

\begin{obs}\label{artobs6}
The factor $|2^n|_F$ has no importance. Indeed, Arthur's arguments involve the adelic functions $\phi= \otimes_v \phi_v$, $f= \otimes_v f_v$. This factor occurs for $\R$ and $\Q_2$, the product of these two factors being equal to $2$. For Arthur's global calculations, we have an identity without these factors.  This suggests, at least in our case, a slight correction in Kottwitz-Shelstad's definition for the transfer factors.
\end{obs}

\begin{obs}\label{artobs6.5}
This theorem remains true if we replace $|det|^s$ by a character $\chi(det)$, where $\chi$ is an abelian character.
\end{obs}

\subsection{Associated elements}\label{artasctele}
Let $\phi \in C_c^\infty(\gnn,\C)$ and $f\in C_c^\infty(\son,\C)$. Fix Haar measures $dg_{2n}$ and $dg$ on the groups $\gnn$ and $\son$ respectively. We say that $\phi$ and $f$ are \emph{associated} (or that they have \emph{matching orbital integrals}) if 
\[
SO_{t^\prime}(f)= \sum_{t\to t'}\Delta(t,t')O_{\theta t}(\phi)
\]
This means that  the stable orbital integral of $f$ is equal to the sum of  the twisted orbital integrals of $\phi$ times the transfer factor $\Delta(t,t')$ defined in chapter $4$ of \cite{K-S}. We suppose also compatibility between the measures and the twisted stabilizers (see \cite{K-S}, page $71$). The next results in this section shall show that conjugacy is equivalent to stable conjugacy.  In the next section, we shall show that the term $\Delta(t,t')$ has a simple expression in our case, more precisely, we shall prove that $\Delta(t,t')= \Delta_{IV}(t,t')$.

We shall use the notation of Kottwitz and Shelstad (\cite{K-S}, chapter 3) for describing the correspondence (modulo conjugacy) between elements  of $\son$ and elements of $\gnn$ (modulo $\theta$-conjugacy). As stated before, we need only the norm for $\theta$-regular and semi-simple elements, therefore we give the complete description only in this case.

It is sometimes more useful to work with a slightly different outer automorphism. In this section we denote by $\theta_0: \gnn \to \gnn$ the map given by $g \mapsto J_0\, ^t g^{-1}J_0$, where the matrix $J_0$ is
\[
J_0= \begin{pmatrix}   & & & &-1\\ & & & 1 & \\ & & \Ddots & & \\ & -1 & & & \\ 1 & & & &   \end{pmatrix}= \begin{pmatrix}
 & K_0\\ K_0 & \end{pmatrix}
\]
This outer automorphism preserves the Whittaker model (see \cite{C-C}).

We shall describe the conjugacy classes in $\son(\overline{k})$ and the $\theta_0$-conjugacy classes  in $\gnn(\overline{k})$. In the case of $\theta_0$-conjugacy we denote, as Kottwitz-Shelstad  do,  by $Cl_{ss}(\gnn,\theta_0)$, the set of classes of $\theta_0$-conjugacy of the group $\gnn(k)$. According to Kottwitz-Shelstad (\cite{K-S}) page $26$,  we have an explicit bijection
\begin{equation}\label{art6eq13}
Cl_{ss}(\gnn,\theta_0)\to T_{\theta_0}^{2n}/ \Omega^{\theta_0}
\end{equation} 
Here $T^{2n}\simeq \G_m^{2n}$ is the diagonal torus and $T_{\theta_0}= T/(1-\theta_0)T$. The set $\Omega^{\theta_0}$ is the centralizer of $\theta_0$ in the Weyl group  $W_{\gnn}= \mathfrak{S}_{2n}$. Thus, $\Omega^{\theta_0}= \mathfrak{S}_{n}\rtimes (\Z/2\Z)^n$ acting on $\G_m^n$ as the extended Weyl group of type $\cald_n$. If $x= (x_1, \dots, x_{2n})$ is a general element of $T$ we have
\[
\theta_0(x)= (x^{-1}_{2n}, \dots, x^{-1}_1)
\]
and 
\[
1-\theta_0(x)= (x_1 x_{2n}, \dots, x_{2n} x_1)
\]
Hence, the image is
\[
\{(u_1, \dots, u_n, u_n, \dots, u_1\}\quad \text{and}\quad T_{\theta_0}= T/Im(1-\theta_0) \simeq \G_m^n
\]
under the map $(x_1, \dots, x_{2n}) \mapsto (x_1/ x_{2n}, \dots, x_n/ x_{n+1})$.  
On the other hand, we have
\[
Cl_{ss}(\son)= \G_m^n/W_{\son},
\]
where $W_{\son}= \mathfrak{S}_n\rtimes (\Z/2\Z)^n$ is the Weyl group of $\son$.
We define as Arthur does $\widehat{SO}_{2n}$  ($\widehat{SO}_{2n} \subset \widehat{GL}_{2n}= GL_{2n}(\C)$) as the centralizer of the element $s = \left( \begin{smallmatrix} -1_n & \\ & 1_n \end{smallmatrix} \right) \in \widehat{SO}_{2n}$. Thus, the matrix defining the group $\widehat{SO}_{2n}$ is
\[
J'= \begin{pmatrix} & -K_0 \\ K_0 & \end{pmatrix}
\]
The torus $T= T_{2n}$ is diagonal, as $\widehat{T}$, so we have
\[
\widehat{T}_{\theta_0} = (\widehat{T})^{\widehat{\theta}_0} = \{ diag(t_1, \dots, t_n, t^{-1}_n, \dots, t^{-1}_1) \} = \widehat{T}_{\son}
\]
We have the isomorphism $(\widehat{T})^{\theta_0} \simeq \widehat{T}_{\son}$ and thus, we have a canonical isomorphism
\[
T_{\son} \simeq T_{\theta_0}
\]
We use now the following proposition.
\begin{prop}[K-S, theorem $3.3.4$]\label{6prp1}
There exists a canonical map 
\[
\cala= \cala_{\son/ \gnn}:Cl_{ss} ( \son) \to Cl_{ss}(\gnn, \theta_0)
\]
This action is defined over $k$, i.e., equivalently  it commutes with the action of the Galois group $Gal(\overline{k}/k)$.
\end{prop}

The inverse mapping $T_{\theta_0} \to T_{\son}$ has a simple form, namely $(x_1,\dots, x_{2n}) \to (x_1 x_{2n}^{-1},\dots, x_{2n}x^{-1}_1)$. In particular, it  is a bijection
\[
T \ni (x_1, \dots, x_n, 1, \dots, 1)  \mapsto (x_1, \dots, x_n, x^{-1}_n, \dots, x^{-1}_1) \in T_{\son}.
\]
The inverse map is called the \emph{norm} $\cala$. We have,
\[
\cala: \G^n_m /W_{\son} \to \G^n_m/\Omega^{\theta_0}
\]
where $W_{\son}$ is the Weyl group and $\Omega^{\theta_0}$ is the extended Weyl group.

We recall that $\son$ is defined by the matrix $J'$ (which indeed defines the split group of type $\cald_n$).  For $g\in GL_n$ let $g'$ be defined by $g'= K_0\, ^tg^{-1} K^{-1}_0$. We consider now the following subgroups,
\begin{enumerate}
\item $\son \supset M = \{ g, g'\}$, a Levi subgroup of $\son$, $g \in GL_n$.
\item $\gnn \supset L = GL_n\times GL_n$ (embedded diagonally).
\end{enumerate}

We have now the following lemma.
\begin{lema}\label{rt6lem1}\hfill \\
\begin{enumerate}
\item Let  $T$ be a Cartan subgroup of  $\son$ contained in  $L$. Then $H^1(k, T)= \{1\}$.
\item In particular, if  $\gamma \in T$ is a regular element, every element that is stably conjugate to $\gamma$ (in $\son$) is conjugate to  $\gamma$.
\end{enumerate}
\end{lema}
\begin{proof}
Part $1$ is Hilbert's theorem $90$ ($T$ is the product of tori $Res_{K/k} \G_m$, where $K$ is an extension of $k$). Part $2$ follows from $1$ since the conjugacy classes in a stable regular conjugacy class are parametrized by
\[
Ker(H^1(k,T)\to H^1(k, \son))
\]
\end{proof}

Next, we recall a calculation made in the article of Chenevier-Clozel (\cite{C-C}). In $L$, the automorphism $\theta_0$ is given by
\[
(g_1, g_2)  \mapsto (g_2', g_1').
\]
The map $I$ given by $(g_1, g_2) \mapsto (g_1, g'_2)$ conjugates $\theta_0$ with the map $\sigma$ defined by
\[
(g_1, g_2) \mapsto (g_2, g_1).
\]
This means that the Galois conjugacy in the trivial case of one split extension $k\times k / k$, a case already known by the work of Clozel and Arthur in base change (see \cite{A-C}). Thus, we deduce the following lemma. 

\begin{lema}\label{art6leme3}\hfill \\
\begin{enumerate}
\item If $\delta \in L(k)$, then $\delta$ is strongly $\theta_0$-regular (i.e. its twisted centralizer is a torus) if and only if $\delta \delta^\theta_0 \in \gnn$ is regular.
\item Every element strongly $\theta_0$-regular in $\gnn$ is $\theta_0$-conjugate to an element of the form
\[
\delta= (g_1, 1),
\]
$g_1 \in GL_n$, regular.
\item For  $\delta$ strongly $\theta_0$-regular in  $L$, $\theta_0$-conjugacy and stable $\theta_0$-conjugacy coincide.
\end{enumerate}
\end{lema}
\begin{proof}
Parts $1$ and $2$ reduce to straightforward calculations in the case of split base change. If $\delta$ is of type $2$, its $\theta_0$-centralizer can be calculated with the bijection $I$ defined above and we have
\[
(x_1,x_2) \to (x_1, x_2)^\sigma (g_1, 1)(x_1, x_2)^{-1}= (g_1,1)
\]
Thus $x_1= x_2^{-1}$, $x_2g_1x_1= x_1^{-1}g_1 x_1= g_1$ and so this set is isomorphic to the centralizer of  $g_1$. Part $3$ is proved as in lemma \ref{rt6lem1} since the classes of stable $\theta_0$-conjugacy are parametrized by
\[
Ker\left(H^1(\gnn^{\delta,\theta_0} \to H^1(\gnn))\right)
\]
where $\gnn^{\delta,\theta_0}$ is the twisted centralizer. In that case, both groups $H^1$ are trivial.
\end{proof}

\begin{obs}\label{artobs8.9}
It is possible that, for all elements the notions strongly regular and regular coincide. We consider only the norm for the calculation of characters. So, it suffices to consider the dense set of elements strongly $\theta_0$-regular. Finally, only  elements having the form $2$ of lemma \ref{art6leme3}, occur in the calculation of the character. The norm calculated in these elements has a very simple form.  
\end{obs}

\begin{lema}\label{artformanorma}
Let $T_1\subset GL_n$ be a maximal torus. The norm   $\cala$, mapping the torus $T_1\times T'_1$ (where $T'_1= K_0\, ^t T_1 K_0^{-1}$) into  $T_1\times \{1\}\subset L $, is given by
\[
(t_1, t'_1) \mapsto (t_1,1)
\]
\end{lema}
\begin{proof}
Since conjugacy and stable conjugacy are the same, it is enough to prove it over $\overline{k}$, for the diagonal torus. Thus we have
\[
(x_1, \dots, x_n, x^{-1}_n, \dots, x^{-1}_1) \mapsto (x_1, \dots, x_n, 1)
\]
as claimed.
\end{proof}

\subsection{Calculation of the transfer factor}
\subsubsection{Factor $\Delta_I$}
We shall describe $\Delta(\gamma, \delta)$. First, we recall the definition of $\Delta(\gamma, \delta)$ given by Kottwitz and Shelstad (\cite{K-S} pages $31-33$). In this case, we set $T_{\son}$ to be the centralizer of $\gamma$ in $\son$, which is a torus. Let $T_{2n} \subset \gnn$ be a torus. We denote by $\gnn^{sc}$, the simply connected covering space of the subgroup $(\gnn)_{der}$, the derived subgroup of $\gnn$. Thus, we have
\[
GL^{sc}_{2n}= SL_{2n}
\]
The morphism  $\theta^*_0$ defined in \cite{K-S}, page $31$  and the morphism $\theta_0$ are the same. We shall denote by $GL^{x}_{2n}$ (see \cite{K-S} \S $4.2$ page $31$) the group of fixed points of the automorphism $\theta_0$ in $GL^{sc}_{2n}$. Therefore, $GL^{x}_{2n}= SL_{2n}^{\theta_0}$. Since $\theta_0$ is described  by the matrix $J_0$, that is anti-symmetric, we have
\[
GL^{x}_{2n}= Sp_{2n}
\]
According to lemma $3.3$B of \cite{K-S} (page $28$), we choose a pair $(B,T)$ for $\gnn$ such that the norm $\cala : T_{\son} \to T/(\theta_0-1)T$ is defined over $k$ (and $T$  is $\theta_0$-stable). We shall fix the choice of this morphism such that 
\[
T=\{ (t_1,t'_1)\}\subset GL_n\times GL_n.
\]
$B$ is a Borel subgroup over $\overline{k}$ adapted to $T$ and $t'_1= K_0\,^tt_1^{-1}K_0$. Thus, if we consider the simply connected part, we obtain
\[
T^{sc}=\{( t_1,t'_1) \mid det(t_1,t'_1)= 1 \}=T
\]
We have $T^x= T^{sc}\cap GL^x_{2n}= T \cap Sp_{2n}$. Hence $T^x$ is a maximal torus of $M^x$ which is a Levi subgroup of $Sp_{2n}$. We have, as previously,  $H^1(k,T^x)={1}$. According to \cite{K-S} page $33$
\[
\Delta_I(\gamma,\delta)= \left< a, s \right>
\]
where $a \in H^1(k,T^x)$ and the element $s$ is deduced from $(T,\theta_0)$. Finally, we obtain
\[
\Delta_I(\gamma,\delta)=1
\]

\subsubsection{Factor $\Delta_{II}(\gamma,\delta)$}
We consider the torus $T= T_1\times T_1 \subset \gnn= \gnn^*$ and we consider also the orbits of $Gal(\overline{k}/k)$ on the system of restricted roots $R_{res}(\gnn,T)$, following \cite{K-S} \S $1.3$. This restriction gives
\[
R_{res}(\gnn,T)= \{ \alpha |_{T^{\theta_0}} \mid \alpha \in R(\gnn,T)\}
\]
We have $T^{\theta_0}= \{(t,t')\}$. Next, we calculate the roots over $\overline{k}$ and then
\begin{equation}\label{art6eqnokase}
T^{\theta_0}= \{x_1,\dots, x_n, x_n^{-1}, \dots, x_1^{-1}\}
\end{equation}
Thus the restricted roots are
\[
\begin{cases}  x_i\pm x_j \quad(i\neq j)\\ \pm 2x_i \end{cases}
\]
This set  is characterized by $R_{res}= R(\son^{\theta_0},T^{\theta_0})= R(\son^x, T^{\theta_0})$. Note that $R_{res}$ is a reduced root system. Now we use lemma $4.3.A$ from \cite{K-S}. We consider the orbits of the roots that do not appear in the group $\son$. In our case, these are the roots $\pm 2x_i$. We recall (see \cite{K-S} \S 1.3) that this term in defined in terms of the theory of Langlands-Shelstad \cite{L-S}. For any root system $\calr$ provided with an action of $\Gamma_k=Gal(\overline{k}/k)$ (in our case $R_{res}(\son, T)$), Langlands and Shelstad associate the data
\[
(\chi_{\alpha_{res}}, \,a_{\alpha_{res}})\qquad \alpha_{res}\in R_{res}
\]
Here $\alpha_{res}$ define a finite extension $k_{\alpha_{res}}/k$ such that $Gal(\overline{k}/k_{\alpha_ {res}})= \{ \sigma \in \Gamma_k\mid \sigma \alpha_{res}= \alpha_{res}\}$, the term $a_{\alpha_{res}}\in k^\times_{\alpha_{res}}$ and $\chi_{\alpha_{res}}$ is a character of $k^\times_{\alpha_{res}}$ (see \cite{L-S} \S$2.2$ and \S$2.5$).  However, the elements $a_{\alpha\, res}$ are not really important in our case. Recall that
\[
(\chi_{\alpha_{res}}, \,a_{\alpha_{res}})\qquad \alpha_{res}\in R_{res}
\]
is formed by the roots $(\pm x_i\pm x_j,\, 2x_i)$, where $T^{\theta_0}$ is described by \eqref{art6eqnokase} in $\overline{k}$. In particular, since in this case $T$ is a product of induced tori, $\prod k_i^\times$, where $k_i/k$ are extensions (embedded twice in $\gnn$), the Galois group $\Gamma_k$ permutes the roots $(\pm x_i\pm x_j)$ and $(\pm 2x_i)$. In particular, $\Gamma_k$ permutes the short roots and permutes separately the roots $2x_i$ and $-2x_i$. If 
$\alpha_{res}= 2x_i$, we see that the stabilizers of $\{\alpha_{res}\}$ and $\{\pm \alpha_{res}\}$ in $\Gamma_k$ coincide. We shall denote this group by $\Gamma_{\alpha_{res}}$. We have now the notion of "$\chi-data$" for any root $\alpha_{res}\in \calr$ (defined by \cite{L-S}, \S$2.5$), satisfying the following conditions
\begin{enumerate}
\item $\chi_{-\alpha_{res}}= \chi^{-1}_{\alpha_{res}}$.
\item $\chi_{\sigma\alpha_{res}}= \chi^\sigma_{\alpha_{res}}$,  for $\alpha_{res}\in \calr$ and $\sigma \in \Gamma_k$.
\end{enumerate}
There is also another condition (\cite{L-S}, \S $2.5$(iii), page $235$), but in our case it is not important. Finally, we have that the characters  $\chi_{\alpha_{res}}$ are independent. In particular, the family given by $\chi_{\alpha_{res}}=1$ (long roots), extended in an arbitrary way for the short roots, verifies the conditions. Thus, for any choice of factors $a_{\alpha_{res}}$ the factor $\Delta_{II}$, according to lemma $4.3.A$, is given by
\[
\Delta_{II}= \prod_{\alpha_{res},\,\text{long}} \chi_{\alpha_{res}}\left(\frac{\alpha(\delta)-1}{a_{\alpha_{res}}}\right)
\]
and is equal to $1$.

\subsubsection{Factor $\Delta_{III}$}
This factor is the most difficult to define. However, in our case it simplifies since $\gnn$ is split and $\theta_0$ fixes the torus and the standard Borel subgroup. The definition of  $\Delta_{III}$ appeals to the hypercohomology of a complex of two chains $f:A \to B$ of $\Gamma_k$-modules, where $\Gamma_k$ is the absolute Galois group of $k$, or it could be also the Weil group $W_k$, see the appendix $A$ in \cite{K-S}.
We recall that 
\[
\xymatrix{H^1(\Gamma_k,  A \ar[r]^-f & B)= Z^1/B^1,
}
\]
where 
\[
Z^1=\{ (c,b)\mid c\in Z^1(\Gamma_k,  A),\, b\in B,\,\partial b= f(c)\}
\]
and the differential operator $\partial:A\to Z^1(\Gamma_k, A)$ is given by $\partial(a)(\sigma)= a^{-1}\sigma (a)$, $\sigma\in \Gamma_k$. The set of coboundaries  $B^1$ is given by
\[
B^1=\{ (\partial a, f(a))\mid  a\in A\}
\]
We have also a long exact sequence
\begin{equation}\label{6eqhyper1}
\xymatrix{ \cdots \ar[r] & H^0(\Gamma_k, A) \ar[r]^f & H^0(\Gamma_k, B)  \ar[r]^j & H^1(\Gamma_k,  A\ar[r]^-f & B) \ar[r]^-i &\\
  & \ar[r]^-i & H^1(\Gamma_k, A) \ar[r]^f & H^1(\Gamma_k, B) \ar[r]  & \cdots & 
}
\end{equation}
where the mappings $i:(c,b) \to $ class of $c \in H^1(\Gamma_F, A)$ and $j: c' \mapsto (0,c')$, $c'\in H^0(\Gamma_F,B)\subset B$. In our case, the constructions are simplified.

Let $T\subset \gnn$ be a maximal $\theta_0$-stable torus, contained in a  $\theta_0$-stable Borel subgroup (non necessarily defined over $k$) denoted by $B_{\overline{k}}$. Consider $T_{\theta_0}= T/(1-\theta_0)T$ and $N:T\to T_{\theta_0}$ the quotient map. Let $\delta \in \son(k)$ and $\overline{\gamma} \in T_{\theta_0}(k)$. We say that $\overline{\gamma}$ is a norm of $\delta$ if $g^{-1}\delta\theta_0 g= t$,  $g\in \gnn(\overline{k})$,  $t\in T(\overline{k})$ and $N(t)= \overline{\gamma}$.  We have $\theta_0$ such that 
$g\mapsto J_0\, ^tg^{-1} J_0:= g'$, and $J_0$ is our known matrix
\[
J_0= \begin{pmatrix} & & & & -1\\ & & & 1 &\\ & & \Ddots & & \\ & -1 & & &\\ 1 & & & &\end{pmatrix}= \begin{pmatrix} & K_0\\ K_0 & \end{pmatrix}
\]
This matrix defines a symplectic group. Recall that $\son$ is defined by the matrix
\[
J'= \begin{pmatrix} & -K_0\\ K_0 & \end{pmatrix}
\]
The previous calculations of the characters remain true by a change of variables. We have in $\son$ a Levi subgroup $M'$
\[
M'= \begin{pmatrix} g & \\ & g' \end{pmatrix}. 
\]
Recall that $g'= J_0\, ^tg^{-1} J_0$. Its dual group 
\[
\widehat{M'} \subset \widehat{SO}_{2n} \subset GL_{2n}(\C) = \widehat{GL}_{2n}
\]
is given by the same equations. On the other hand, Kotwittz and Shelstad  (\cite{K-S} page $4$) define
\[
\widehat{G}^1= \widehat{GL}_{2n}^{\widehat{\theta}_0}= Sp(J_0, \C)
\]
The morphism $\widehat{\theta}_0$ identifies with $\theta_0$ and 
\[
^L\widehat{G}^1 \simeq \widehat{G}^1 \times W_k
\]
In particular $^L\widehat{G}^1$ and $^L\widehat{SO}_{2n}$ have the common subgroup 
\[
\widehat{M'}= \left\{ \begin{pmatrix} g & \\ & g' \end{pmatrix} \right\}\times W_k
\]
If $T= T_1 \times T'_1$ is one of the tori we consider, the morphisms $\theta_0$ and $\widehat{\theta}_0$ coincide on $T$. We have $T_{\theta_0}\simeq T_1$. According to the definition of the norm $\cala$,  an element $\gamma \in T_{\theta_0}(k)= T_1(k)$ is the norm of $\delta$ if $\delta= \cala(\gamma)$. In particular, for $\gamma$ strongly regular, $\delta$ ($k$-rational element) exists and is unique up to $\theta_0$-conjugacy ($k$-rational). From this data, Kottwitz and Shelstad  define 
\[
c= c_\sigma \in Z^1(\Gamma_k, T(\overline{k}))
\]
with the condition
\[
c_\sigma = g\sigma (g)^{-1},
\]
which is equal to $1$ in our case, since $g=\delta$ is rational. We have now
\[
c_\sigma  \theta_0(c_\sigma)^{-1}= 1= \gamma \sigma(\gamma)^{-1}
\]
Setting $A=B=T$, in the above construction,  $T=T(\overline{k})$ and $f= 1-\theta_0$, we remark, that according to the definitions, the term
\[
(c_\sigma=1,\gamma) \in Z^1(\Gamma_k,T)\times T
\]
is a hypercocyle, but it is not a hypercoboundary. It defines, thus, an element 
\[
\xymatrix{\V \in H^1(\Gamma_k,\, T  \ar[r]^-{\,\,1-\theta_0} & T)
}
\]
Let $\widehat{T}\subset \widehat{GL}_{2n}$ be  the dual torus of $T$ (provided with the action of $W_k$ or $\Gamma_k$ since $T_1 = T_1= K_1^\times \times \cdots \times K^\times_r$, where $K_i/k$ are field extensions). We have $\widehat{T}= \widehat{T}_1\times \widehat{T}_1$, where $\widehat{T}_1\subset GL_n(\C)$ and the torus $T_G$ of $\son$ is isomorphic to $T_1$. According to Langlands-Shelstad \cite{L-S}, we can construct naturally a homomorphism of $L$-groups
\[
\xymatrix{^LT_1= \widehat{T}_1\rtimes W_k \ar[rd] \ar[r]  &  ^LGL_n= GL_n(\C)\rtimes W_k \ar[d]\\ & W_k
}
\]
We denote by $\xi^1$ this morphism. We write $\xi^1(\hat{t}_1,w)= (\xi^1(\hat{t}_1,w),w)$, by  abuse of notation. Thus, the morphism $\xi: ^LT_G= \widehat{T}_1\rtimes W_F\to\, ^L\gnn$, is given by 
\[
(\hat{t}_1,w) \mapsto \left(\begin{pmatrix} \xi_1(\hat{t}_1,w) & \\ & \xi'_1(\hat{t}_1,w) \end{pmatrix},\,  w\right)
\]
It maps  $^LT_{G}$ into the Levi $L$-subgroup common to $^L\widehat{SO}_{2n}$ and $^L G_1$. The two morphisms 
$\xi:\, ^L T_{G}\to \, ^L\widehat{SO}_{2n}$ and $\xi_1:^L T_{G}\to \,^L G^1$ (defined in \cite{K-S} page $4$)  coincide in this case. Their difference $\xi \xi_1^{-1}:\,^L T_{G }\to \, ^L\widehat{SO}_{2n}$ defines a $1$-cocycle $A\in Z^1(W_k, \widehat{T})$. Thus $A=1$ in our case. 

On the other hand, by construction, $^L\widehat{SO}_{2n}$ and $\widehat{G}^1$ appear naturally as $\theta_0$-centralizers of  $\theta_0$-semi-simples elements $s \in \gnn(\C)$.  Since $\theta_0$ is symplectic, the group $\widehat{G}^1$ is simply defined by $s=1$. It is equal to the set of elements of $\widehat{GL}_{2n}$ such that 
\[
(g,\theta_0)1(g,\theta_0)^{-1}= gJ_0\, ^tg^{-1}J_0^{-1}= 1
\]
If we consider now
\begin{equation}\label{6eq0191}
s= \begin{pmatrix} 1 & \\ & -1\end{pmatrix}
\end{equation}
such that $J=J_0 s$, the twisted endoscopic group becomes $\widehat{SO}_{2n}$. Now Kotwittz and Shelstad consider the pair 
\[
(A,s)=(1,s)\in Z^1(W_k,\widehat{T})\times \widehat{T}
\]
We have now $f(1)=1= \partial s= s^{-1}\sigma(s)$, $\sigma \in W_k$. Since $s$ is given by equation \eqref{6eq0191}, it belongs to $\G_m\times \G_m\subset \widehat{T}_1\times \widehat{T}_1$.  If $T_1  = \prod K_i^\times$, $\G_m$  is the $L$-group of $k^\times$ (for the map $T_1 \to \G_m/k$, given by the product of norms). Thus the groups $W_k$ or $\Gamma_k$ act trivially. Therefore
\[
(c=A=1,s)\in Z^1(W,\widehat{T})\times \widehat{T}
\]
is a hypercocycle, but not a hypercoboundary. Thus, it defines an element
\[
\xymatrix{\A \in H^1(W_k,\,\widehat{T}  \ar[r]^-{\,\,1-\theta_0} & \widehat{T})
}
\]
The factor $\Delta_{III}$ is then defined as
\begin{equation}\label{6eq0201}
\Delta_{III}(\gamma,\delta)= \left<\V, \A\right>,
\end{equation}
given a Tate  pairing  in hypercohomology (this is the main object in the definition of \cite{K-S}, appendix $A$). Recall that neither $\V$ nor $\A$ are trivial classes. 

Now an isotropy result, assures the triviality. Consider the following exact  sequence
\begin{equation}\label{6eqhyper2}
\xymatrix{ \cdots \ar[r]  & H^0(\Gamma_k, T) \ar[r]^-j & H^1(\Gamma_k,  T \ar[r]^-f & T) \ar[r]^-i & H^1(\Gamma_k, T) \ar[r]^-f   & \cdots
}
\end{equation}
where $f= 1-\theta_0$. Since $T=T_1\times T'_1$ is a reduced torus, we have
\[
H^1(\Gamma_k, T)= \{1\}
\]
Now, the have the following lemma:
\begin{lema}\label{6lemmanos4}
For the element $\V$, we have
\[
\xymatrix{\V \in Im\Big(H^0(\Gamma_k,T) \ar[r]^-j & H^1(\Gamma_k,\,T \ar[r]^-f & T)\Big)
}
\]
\end{lema}
Note that  neither $H^0(\Gamma_k,T)$ is  trivial, nor the mapping $(1-\theta_0)$ is not injective. In a similar way, for $\A$, we have the following exact sequence in hypercohomology:
\begin{equation}\label{6eqhyper3}
\xymatrix{ \cdots \ar[r]  & H^0(W, \widehat{T}) \ar[r]^-{\widehat{j}} & H^1(W,  \widehat{T} \ar[r]^-{1-\hat{\theta_0}} & \widehat{T}) \ar[r]^-{\widehat{i}} & H^1(W, \widehat{T}) \ar[r]  & \cdots
}
\end{equation}
where $\widehat{i}$ is the map $(c,\widehat{t})\to $ class of $c$ and $\widehat{j}$ is the map such that $c' \to (1,c')$. Here we have $(c,\hat{t})= (1,s)$, $\hat{i}(\A)= 1$. We have now the following proposition:
\begin{prop}[Kottwitz-Shelstad, page $137$]\label{6propkillbill}
For all elements $u \in H^0(\Gamma_k, T)$, $\hat{z}\in H^1(W_k,\widehat{T} \to \widehat{T})$
\[
(j(u),\widehat{z})= (u,\widehat{i}(\widehat{z}))
\]
where the product on the right side is the Langlands pairing  (\cite{L3}, \cite{Lab2})
\[
H^0(\Gamma_k,T)\times H^1(W,\widehat{T}) \to \C^\times
\]
\end{prop}
Since $\V$ is in the image of $j$ and $\widehat{i}(\A)=1$, we deduce that $\Delta_{III}(\gamma,\delta)= 1$. Finally, we have proved the following theorem:
\begin{teo}\label{art6teokillbill}
Let $T= T_1\times T'_1\subset L$  be a torus and $T_{G}= \{ (t_1, t'_1)\} \subset M$ be the associated torus,  $\delta \in T$, $\gamma \in T_{G}$ associated elements. Then
\[
\Delta(\gamma,\delta)= \Delta_{IV}(\gamma,\delta).
\]
\end{teo}


\section{Local Arthur's packets}
In this section, we study the composition of the Arthur's local packet. We establish this composition from the calculation of characters and also from Arthur's conjectures. We consider $\son$ as an elliptic twisted endoscopic subgroup of $\gnn$. Recall that a twisted endoscopic group $H$ of $\gnn$ is called \emph{elliptic} if $(Z(\widehat{H})^{\Gamma_\Q})^0$ is contained in $Z(\widehat{GL}_{2n}^0)$. Denote by $\cale_{\text{ell}}(\gnn)$, the set of twisted endoscopic groups of $\gnn$. With the aim of describing all the set $\cale_{\text{ell}}(\gnn)$, Arthur decomposes the integer $2n$ as a sum of positive integers
\begin{equation}\label{artartinteg}
2n= N_0+ N_s
\end{equation}
where both $N_0$ and $N_s$ are even. We will say that an element $H \in \cale_{\text{ell}}(\gnn)$ is \emph{primitive} if one of them $N_0$ or $N_s$, are equal to zero. We denote by $\cale_{\text{prim}}(\gnn)$,  the subset of primitive elements in $\cale_{\text{ell}}(\gnn)$. The case of $\son$ split occurs when $N_0= 2n$. We denote by $\widetilde{\Psi}(\gnn)$, the set of parameters of $\gnn$ invariant  by the outer automorphism $\theta$. Likewise, we set $\widetilde{\Psi}(\son)$, to be the parameters in $\widetilde{\Psi}(\gnn)$ factorizing through $^LSO_{2n}$. We denote by $\dsonv$, the local form of the dual group $\dson$. Let $SO_\gamma(f)$ be the stable orbital integral of the function $f$, $f \in C^\infty_c (\son(\Q_v))$ and   $\gamma \in \son$ a strongly regular element. Let $\widetilde{\calh}(\son)$ be defined by
\[
\widetilde{\calh}(\son)= \{ \text{functions on $\son \mid\, SO_\gamma(f)$ is  Out$(\son)$-invariant} \}.
\]
For each parameter $\psi \in \widetilde{\Psi}(\gnn)$ and any function $\phi \in  C^\infty_c (\gnn(\Q_v))$ we consider the local twisted trace:
\[
\text{trace}(\Pi(\psi)(\phi)A_\theta)
\]
where $\Pi$ is the associated representation of $\gnn(\Q_v)$ and $A_\theta$ is the operator that intertwines  $\Pi$ and $\Pi\circ \theta$ (we have by hypothesis $\Pi \simeq \Pi \circ \theta$). This intertwining operator is defined by Arthur (\cite{Ar2}, page $244$).

On the other hand, if we assume that the twisted form of the conjecture stated by  Langlands and Shelstad holds  for the local field $\Q_v$, we obtain an map from the Hecke algebra $\calh(\gnn) \to \widetilde{\calh}(\son)$, given by
\[
\phi \to \phi^{\GO}:= f
\]
where  $\phi$ and $f$ are associated (\emph{cf}. section \ref{artasctele}). We shall calculate some of the objects defined by James Arthur. We recall that in the local case we can identify the local Langlands group $\call_{\Q_v}$ with $W_{\Q_v}$, the absolute Weil group  of $\Q_v$. Thus, we have a parameter given by
\[
\psi:W_{\Q_v}\times SL_2(\C) \to SO_{2n}(\C) \hookrightarrow GL_{2n}(\C)
\]
This local parameter, in our case, has the particular form
\begin{equation}\label{elnuestrito}
\psi= \mu\otimes sp_n \oplus \mu^{-1} \otimes sp_n
\end{equation}
Recall that $sp_n$ is the representation of degree $n$ of $SL_2(\C)$ (since $n$ is even in this case it is symplectic) and $\mu$ is the character $x\mapsto |x|^s$, $s\in \C$. Therefore, the representation associated to this parameter, is the induced representation
\begin{equation}\label{parametrewhitak}
\Pi(\psi)=Ind^{\son}_{GL_n}\left(\mu \circ det \otimes \mu^{-1} \circ det \right).
\end{equation}
We consider normalized induction. The variable $s$ satisfies $|s| < \frac{1}{2}$. Hence, if $\phi$ is a function in $\calh(\gnn)$, the twisted character $\phi_{\gnn}(\psi)$ is defined as:
\[
\phi_{\gnn}(\psi):= trace\left( \Pi(\psi)(\phi)A_\theta \right),
\]
where $A_\theta$ is an intertwining operator naturally associated to $\theta$. For this calculation, we shall rather consider $A_\theta= A_{\theta_0}$, with the matrix $J_0$ fixing the Whittaker model (in fact Arthur uses this theory of Whittaker models to define an extension of this representation to the non-connected group $GL_{2n}\rtimes \Z/2\Z$, see \cite{Ar2}, page $244$). Now we embed this representation \eqref{parametrewhitak} into a larger representation denoted by $R$:
\begin{equation}\label{repchoncha6}
	R= Ind_{GL_n\times GL_n}^{\gnn} \left( Ind_{B_n}^{GL_n}
          (\mu\otimes \cdots \otimes \mu)\delta_{B_n} \otimes
          Ind_{B_n}^{GL_n} (\mu^{-1}\otimes \cdots \otimes \mu^{-1})\delta_{B_n} \right)
\end{equation}
$B_n$ is the usual Borel subgroup of $GL_n$ and $\delta_{B_n}$ is the modular character of $B_n$. This representation $R$ has, as quotient, the representation $\Pi(\psi)$ given by \eqref{parametrewhitak}:
\[
\xymatrix{
R \ar @{->>}[r] & \Pi(\psi)
}
\]
Arthur defines $A_\theta$ as the quotient operator of the operator over $R$, fixing the Whittaker vector and  equal to $1$ in the unramified vector. We use this operator for the quotient representation. We need also the group $\call_\psi$ (pages $240-241$ \cite{Ar2}). Arthur defines a group denoted by $\call_\psi$, associated to any parameter $\psi \in \widetilde{\Psi}(\gnn)$  factorizing through $^L \son$, i.e., $\psi \in \widetilde{\Psi}(\son)$. By the definition of $\widetilde{\Psi}(\son)$, this parameter is self-dual with respect to the automorphism $\theta$. Recall, that in general, such a parameter has the form
\begin{equation}\label{ptravezpsi}
\psi= l_1 \psi_1 \boxplus \cdots \boxplus l_r \psi_r
\end{equation}
Let $\cali$ be  the set of indices $i$ such that $\psi_i^\theta= \psi$. The complement of $\cali$ is made of two disjoint sets $\calj$ and $\calj'$ that are bijective to each other. The bijection is given by the mapping $j \mapsto j'$ such that $\psi_j^\theta = \psi_{j^\prime}$ for every  $j \in \calj$. Thus the parameter $\psi$ can be written as
\begin{equation}\label{losgrandotes}
\psi = \left( \BIGboxplus_{i\in \cali} l_i\psi_i\right) \BIGboxplus
\left( \BIGboxplus_{j\in \calj}l_j(\psi_j \boxplus \psi_{j^\prime}) \right)
\end{equation}
For $i \in \cali$ Arthur applies the global and local induction hypothesis (page $240$ \cite{Ar2}, to the automorphic cuspidal self-dual representation $\mu_i$ of $GL_{m_i}$ associated to $\psi_i$. This gives an endoscopic data in $\cale_{\text{prim}}(G_{m_i}\rtimes \theta)$.  If $j \in \calj$ we set $G_j= GL_{m_j}$. We obtain thus a group $G_\alpha$, over the local field $\Q_v$, for $\alpha \in \cali \cup \calj$. Let $^LG_{\alpha}$ be the Galois form of its $L$-group. Arthur now defines the group $\call_\psi$, as the fibered product of these groups over $\Gamma_{\Q_v}$.
\begin{equation}\label{elpinelepsi}
\call_\psi = \prod_{\alpha \in \cali \bigcup \calj}(^L G_\alpha \to \Gamma_{\Q_v})
\end{equation}

We can define this group globally and locally. We shall study the local form. This group $\call_\psi$ is an extension of the Galois group. We suppose that $\psi$ factorizes through $^L\sonv$ and thus, there must exist an $L$-homomorphism $\psi_{\sonv}$:
\[
\psi_{\sonv}: \call_\psi \times SL_2(\C) \to\, ^L\sonv.
\]
Next, we define $\tilde{\psi}$ as the composition $\xi\circ\psi_{\sonv}$, where $\xi$ is the morphism of the twisted endoscopic data. The morphisms $\psi$ and $\xi$ are $\gnn(\C)$-conjugacy classes of homomorphisms, so $\psi_{\sonv}$ is defined up to conjugacy. Thus, $\tilde{\psi}$ defines an embedding 
\[
\widetilde{\psi}: \call_\psi \times SL_2(\C) \to GL_{2n}(\C) \times \Gamma_{\Q_v} = \,^L\gnn
\]
Finally, there is another object we shall use: the group defined as
\[
 S_\psi= S_\psi(\sonv)=
Cent_{\dsonv}(\psi_{\sonv}(\call_\psi \times SL_2(\C))).
\]
Where $Cent$ is the centralizer. We are interested in the quotient group $\boldsymbol{\cals}_\psi
= S_\psi/S_\psi^0 Z(\dsonv)^{\Gamma_{\Q_v}}$ which is a finite abelian group. In this group $S_\psi$ there is a canonical element denoted by $s_\psi$ and defined as
\begin{equation}\label{elesepsi}
s_\psi= \psi_{\sonv}\left(1,\begin{pmatrix} -1 & 0\\ 0 & -1 \end{pmatrix} \right)
\end{equation}
Its image in $\boldsymbol{\cals}_\psi$ is denoted  by $\overline{s}_\psi$.

We go back to our parameter $\psi$ which, as we saw above, has the form:
\[
\psi= \psi_1 \boxplus \psi_2= \mu\otimes sp_n \oplus \mu^{-1}\otimes sp_n
\]
For this parameter $\psi$, one has two possible cases.
\begin{enumerate}
\item $\mu\neq \mu^{-1}$ (we shall call this one, the regular case). Then
  $\psi_1 \neq \psi^{\theta}$. We set $G_1= GL_1$ and $\call_\psi=
  GL_1 \times \Gamma_{\Q_v}$,  and we recall that $\Gamma_{\Q_v}$ is the Galois group of
   $\Q_v$. Consequently we have a morphism
\[
\call_\psi \times SL_2(\C) \to GL_{2n}(\C)
\]
given by 
\[
(z,s)\mapsto z\, sp_n(s) \oplus z^{-1}sp_n(s)
\]
(recall that $sp_n$ is the irreducible representation of  $SL_2(\C)$ of degree $n$).
\item $\mu = \mu^{-1}= \epsilon$, which implies $\epsilon^2= 1$.  Thus the parameter
  $\psi$ has the form
\[
\psi = \epsilon\otimes sp_n \oplus \epsilon\otimes sp_n
\]
In this  case we set $G_0=SO_1=\{1\}$. Therefore, we have an embedding 
\[
\call_\psi=
SO_1 \times \Gamma_{\Q_v}  \to GL_1(\C)\times \Gamma_{\Q_v}
\]
given by
\[
(1,\gamma)\mapsto (\epsilon(\gamma),\gamma)
\]
Now we set $\call_\psi= \{1\}\times \{1\}\times \Gamma_{\Q_v}$ and we have $\call_\psi  \times SL_2 \to \widehat{GL}_{2n}\times \Gamma_{\Q_v}$. Let $\widetilde{\psi}$ be the morphism such that
\[
\widetilde{\psi}(\gamma,s) = \epsilon(\gamma)\,sp_n(s)\oplus \epsilon(\gamma)\,sp_n(s)
\]
Then, the parameter $\widetilde{\psi}$ factorizes through $SO_{2n}(\C)= \dsonv \subset \widehat{GL}_{2n}$.
\end{enumerate} 

Now we calculate the groups $S_\psi$ and $\boldsymbol{\cals}_\psi$ in both cases. The calculation of the centralizer in the first case gives $S_\psi \simeq \C^\times$ and in the second, $S_\psi \simeq SL_2(\C)$. Next, in both  cases we find that
\[
 \boldsymbol{\cals}_\psi = \{1\} 
\]
From this follows that the element $\overline{s}_\psi \in \boldsymbol{\cals}_\psi$ is obviously $1$.

We denote now by $\widetilde{\Pi}(\son)$, the set of $Out_{\gnn}$-orbits in the set $\Pi(\son)$ of irreducible representations of $\son$. We denote by $\widehat{\Pi}_{\text{fin}}(\son)$, the set of formal finite linear combinations of elements of $\widetilde{\Pi}(\son)$. Any element $\pi \in \widehat{\Pi}_{\text{fin}}(\son)$ determines a linear form on $\widetilde{\calh}(\son)$: 
 \[
 f \mapsto f^{\GO}(\pi)\quad\quad \text{(its trace)}
 \]
 With all this data Arthur in his theorem $30.1$ (\cite{Ar2} page $246$) affirms that for any local field $\Q_v$, $G$ an elliptic twisted endoscopic group of $GL_{2n}$ and for every parameter $\psi \in \widetilde{\Psi}(G)$, there exists a stable linear form $f \to f^G(\psi)$  on $\widetilde{\calh}(G)$ such that, for $\phi, f$  associated functions
\[
\text{trace}\,(\Pi(\phi)I_\theta)= f^G(\psi).
\]
This theorem also establishes that for any parameter $\psi \in \widetilde{\Psi}(G)$, there exists a finite subset $\widetilde{\Pi}_\psi$ of the set of non-negative linear combinations of unitary representations of $G$ modulo the group of outer automorphisms (in our case isomorphic to $\Z/ 2\Z$) and an injective map
\begin{equation}\label{pigrandotapsi}
\pi \mapsto  \left<\,\cdot,\pi  \right>, \qquad \pi \in \widetilde{\Pi}_\psi
\end{equation}
from the set $\widetilde{\Pi}_\psi$ into the group of characters
$\widehat{\boldsymbol{\cals}}_\psi(G)$,  satisfying the following condition:
\begin{equation}\label{6eqult}
f^{\GO}(\psi) = \sum_{\pi \in \widetilde{\Pi}_\psi}
\left<s_\psi,\pi  \right> f_{G}(\pi), \quad f\in \widetilde{\calh}(G)
\end{equation}
where  $s_\psi$  is the element  defined by \eqref{elesepsi} and $f_G(\pi)= trace\,\pi(f)$.

The result in our case is the following:
\begin{teo}\label{7teo1ch}
We consider the split special orthogonal group $SO_{2n}$ and $\psi$, 
the parameter  associated to our representation $\pi= Ind_{GL_n}^{SO_{2n}}(|det|^s \otimes \epsilon)$. Then
\begin{enumerate}
\item There exists a stable linear form  $f \mapsto f^{\son}(\psi)$ on $\widetilde{\calh}(\son)$ such that 
\[
f^{\son}(\psi) \quad \text{is equal to the  trace$(\Pi(\phi)I_\theta) =$  trace$(\pi(f))$}
\]
where $\phi$, $f$ are associated in the sense of stable orbital integrals (\emph{cf.} section \ref{artasctele}).  
\item If $f \in \widetilde{\calh}(\son)$ is unramified, 
\[
f^{\son}(\psi)= \text{trace}\,\pi(f)= \text{trace}\,\pi'(f).
\]
In particular $\pi$ (which is isomorphic to  $\pi'$ in $\widetilde{\Pi}_\psi$) is the only unramified element, of $\widetilde{\Pi}_\psi$, with multiplicity $1$.
\end{enumerate}
\end{teo}
\begin{proof}
Part $1$ is the theorem $30.1$  of Arthur (\cite{Ar2}), in our case  $\boldsymbol{\cals}_\psi= \{1\}$. From part  \eqref{6eqult} we deduce that
\[
f^{\son}(\psi)= f_{\son}(\pi_0)
\]
where $\pi_0$ is the only element of  $\widetilde{\Pi}_\psi$. Combining this with the results of section $5$, we obtain that for $\phi$ and $f$ associated,
\[
f^{\son}(\psi):= trace\,(\Pi(\phi)I_\theta) \quad (\Pi= \Pi(\psi))
\]
\[
= \frac{1}{2}(\text{trace}\,\pi(f) + \text{trace}\,\pi'(f)) = \text{trace}\,\pi(f)= \text{trace}\,\pi'(f),
\]
since the orbital integrals of $f$ are stables under the outer automorphism. If $\phi$ (and thus $f$) is unramified, the equality
\[
\text{trace}\,\pi_0(f) = \text{trace}\,\pi'(f)
\]
shows that $\pi$, with multiplicity $1$, is the only  unramified element of this representation $\pi_0$.
\end{proof}

\begin{obs}\label{obscorreg}
\begin{enumerate}
\item As remarked by Colette M\oe glin to the author,  the norm into $\son$ is surjective. We have, in fact, characterized the  \emph{trace} $\pi_0(f)$ for every function $f \in \widetilde{\calh}(\son)$. This implies that $\pi_0= \pi= \pi'$ (with multiplicity $1$).
\item The calculation remains true in the real  place (or complex in the case of number fields). 
\end{enumerate}
\end{obs}

\subsection{Local irreducibility}
The author thanks Colette M\oe glin deeply for her help in the argument given below. Once we have studied the local objects defined by Arthur, a natural question is whether the local representation associated to our parameter is irreducible.

Arthur's conjectures lead us to consider representations 
\[
W_{\Q_v} \times SL_{2}(\C) \to  SO_{2n}(\C)
\]
having the form
\[
\chi \otimes sp_n \oplus \chi^{-1} \otimes sp_n,
\]
where $\chi$ is an unramified character. If $\chi$ is unitary and $\chi^2 \neq 1$ or $\chi = |\,|^s$ with $s$ real and small, the induced representation is irreducible according to well known results. If $\chi$ has order $2$ ($\chi = \epsilon$ in our notations), then the calculation in the preceding subsection shows that the associated centralizer $S_\psi$  is connected. In that case, Arthur shows that the induced representation has to be irreducible. On the other hand, Langlands gives the result for the parabolic induction, when the inducing representation is a discrete series. There are well known results, in this case, to determine whether the representation
\[
Ind^{\son}_{GL_n}(\delta\otimes s):= Ind^{\son}_{GL_n}(\delta \otimes |det|^s), \quad s \in i\R,
\]
induced from the maximal parabolic subgroup $GL_n$ of $\son$, is irreducible (see for instance \cite{Wal0}). These criteria of irreducibility are:
\begin{enumerate}
\item If $(\delta \otimes s) \not\simeq w_0(\delta \otimes s)$, where $w_0$ is the element of maximal length in 
the Weyl group $W_{\son}$, then the induced representation is irreducible.
\item If $(\delta \otimes s) \simeq w_0(\delta \otimes s)$, consider $s+t$ ($s\in i\R,\,t\in \R$ and $t >0$) and  consider the intertwining operator $A(w_0,t)$
\[
\xymatrix{
Ind^{\son}_{GL_n}(\delta\otimes s \otimes w_0 t) \ar[r]^{A(w_0,t)} & Ind^{\son}_{GL_n}(\delta
\otimes (s+t))
}
\]
If the intertwining operator $A(w_0,t)$ has a pole at  $t=0$, then the induced representation is irreducible.
\end{enumerate}

In our case the inducing representation is not a discrete series.  We induce from the parabolic subgroup containing $M=GL_n$ as Levi subgroup, the representation $Ind(triv\otimes s)= \pi$, with the additional assumption that $\pi$ is everywhere unramified. This implies that all local representations $\pi_v$ are unramified. So this representation is a representation of the Hecke algebra $\calh(\son,I)$, where $I$ is the Iwahori subgroup of $\son$.  We consider now a particular involution of this algebra $h \mapsto h^*$,  $h \in \calh(\son,I)$. This involution has been studied by Marie-Anne Aubert in  \cite{Aub}. One of the main results in this paper is that a representation of this algebra is irreducible if and only if the representation transformed by the involution is also irreducible. It is known that this involution permutes the  trivial representation with the Steinberg representation $St$. We use this involution and we study the Steinberg representation as an induced representation.

The Steinberg representation is embedded in a representation induced from the principal Borel subgroup:
\[
 St \hookrightarrow Ind_B^{GL_{n}}(|\;|^{\frac{n-1}{2}},\dots, |\;|^{\frac{1-n}{2}})
\]
Thus, we have a diagram 
\[
\xymatrix{ Ind^{SO_{2n}}_{GL_n}(St\otimes \epsilon |\;|^{s}) \ar[d]^{A_{w_0}} \ar @{^{(}->}[r] & Ind_B^{SO_{2n}}(|\;|^{\frac{n-1}{2}+s+\epsilon},\dots, |\;|^{\frac{1-n}{2}+s+\epsilon}) \ar[d]^{A_w} \\ Ind(St\otimes \epsilon |\;|^{-s}) \ar @{^{(}->}[r] & Ind_B^{SO_{2n}}(|\;|^{\frac{n-1}{2}-s+\epsilon},\dots, |\;|^{\frac{1-n}{2}-s+\epsilon})
}
\]
where $A_w$ is the "long" intertwining operator  for $B\subset SO_{2n}$.

The operator $A_w$ sends the elements $(x_1,\dots, x_n)$ to $(x_n^{-1}, \dots, x_1^{-1})$. This operator is a product of operators realized in $GL_2$ or $SO_4$. The case of $GL_2$ being simple, we consider only the case of $SO_4$. In that case we have a diagram
\[
\xymatrix{ Ind^{SO_4}_{GL_2}(St\otimes \epsilon |\;|^{s}) \ar[d]^{A_{SO_4}} \ar @{^{(}->}[r] & Ind_{B}^{SO_{4}}(|\;|^{1/2}, |\;|^{-1/2}\otimes \epsilon |\;|^s) \ar[d]^{A_{w,SO_4}} \\ Ind^{SO_4}_{GL_2}(St\otimes \epsilon |\;|^{-s}) \ar @{^{(}->}[r] & Ind_{B}^{SO_{4}}(|\;|^{1/2}, |\;|^{-1/2}\otimes \epsilon |\;|^{-s})
}
\]
The letter $B$ now denotes the usual Borel subgroup of $SO_4$. On the other hand, we know that $SL_2\times SL_2/\pm 1 \simeq SO_4 \supset GL_2 \simeq SL_2\times \G_m/\pm 1$. Thus the problem is reduced to considering the intertwining operator 
\[
Ind_{\calb}^{PGL_2}(\epsilon |\;|^{s}) \to Ind_{B_2}^{PGL_2}
\]
where $\calb$  Borel subgroup of  $PGL_2$. Now we are in Langlands' conditions and we know that this operator has a pole, since it is associated to a quotient of $L$-functions 
\[
 \frac{L(\epsilon^2,2s)}{L(\epsilon^2, 2s+1)}
 \]
with $\epsilon^2=1$, which has a pole. Since the operator $A_{SO_4}$ coincides, on the spherical vector, with the operator $A_{w_0}$, we deduce that it has a pole and therefore the representation $St$ is irreducible. In view of Aubert's result, we obtain finally that $\pi= Ind_{GL_n}^{SO_{2n}}(\epsilon \circ det)$ is irreducible.
 \begin{obs}
 We have the same result in the real place for unramified parameters.
 \end{obs}
 These calculations allow us to conclude this proposition:
 \begin{prop}\label{moeirrerep}
 The local representation $Ind_{GL_n}^{SO_{2n}}(\chi)$ ($\chi$ unitary or $\chi = |\,|^s$, $s$ real and small) associated to parameters of  type $\chi\otimes sp_n$ for the split group $SO_{2n}$ is irreducible.
 \end{prop}
 
 
 \section{Global results}
 In this section we establish some consequences in the case of the group $\son$ for the parameter we have studied. We remain in the case everywhere unramified.
 
 We consider our parameter of section $2$, $\psi \in\widetilde{\Psi}_2(\son)$: it is a global square  integrable parameter of $\son$. We set $\psi_\infty$ the archimedian component, and let $f \in  \widetilde{\calh}(\son)$, a function everywhere unramified. For this function we have also a decomposition $f = (f_\infty, f_p)$, where all the $f_p$ are unramified. If we consider the calculation of the trace for $f$,  we find (see also \cite{Ar2}, page $247$) that:
 \[
\text{trace}\,\left(f| L^2_{\psi_\infty}\right)=\sum_{\substack{\pi \in L^2_{\text{disc}}\\ \omega_\pi = \omega_{\psi_\infty}}} \,\text{trace}\,\pi(f)= \sum_{\substack{\psi^\prime \neq \psi\\ \psi^\prime_\infty = \psi_\infty}} m_\psi \bigotimes_{\pi \in \calb_\psi} \pi_v + 2 \prod_v \text{trace}\, (\otimes \pi_v) (f)
\]
We consider a finite set of places $S$ and $f_S$, any  function associated to this set $S$.  We verify by varying the function $f_S$ and then the set $S$, that all the previous terms are linearly independent from last term in the previous equation.  This implies that the trace is simply
\[
\text{trace}\,\left(f | L^2_{\psi_\infty, \psi_S}\right)= 2 \prod_v \text{trace}(\otimes \pi_v)(f)
\]
Let $\pi$ be the residual representation of $\son$ of multiplicity one in $L^2_{\psi_\infty}$, constructed in section $2$. Let $\alpha' = \pi^\alpha$ be the image of $\pi$ under the outer automorphism $\alpha$. By transport of structure, the representation $\pi'$ is also residual, with multiplicity one but these representations are not isomorphic (Langlands theorem of disjunction for parabolic subgroups). Moreover, we know that, locally, the Hecke matrices of these two representations are different (\emph{cf}. remark \ref{obsmatrixhecke}). Consequently,
\[
\text{trace}\,(f| L^2_{\psi_\infty,\psi_S})= \sum_{\substack{\pi \in
    L^2_{\text{disc}}\\ \text{type $\psi_\infty,\psi_S$}}}
\text{trace}\,\pi (f)= \text{trace}\,\pi(f)+\text{trace}\,\pi'(f) +
\sum\,' \,\text{trace}\,\rho(f)
\]
\[
= 2\prod_v \text{trace}\, \pi(f)
\]
where $\rho$ denotes the representations of $\widetilde{\calh}(\son)$,  different from $\pi$, that might appear with the type $\psi$. Since trace $\pi =$ trace $\pi'$, on $\widetilde{\calh}(\son)$, the previous identity, true for all the unramified functions, shows that the sum $\Sigma'$ is empty. Thus we have this result.

\begin{teo}\label{8teo1ja}
Assume  the global and local results of Arthur, chapter
30 \cite{Ar2}. There exist only two representations of type  $\psi$
(for our fixed parameter $\psi$) in the space $L^2\left(SO_{2n}(\Q)
  \backslash SO_{2n}(\A_\Q)\right)$. These are the residual representations  $\pi$ and $\pi'= \pi^\alpha$ where $\alpha$ is the  outer automorphism  of $SO_{2n}$ (notations as above). In particular, there is no cuspidal representation of  $SO_{2n}(\A_\Q)$ of type $\psi$.
\end{teo}

Finally, we can  ameliorate the bounds of the eigenvalues for the  Hecke operators given in section $4$. Using these estimates we deduce now the following theorem:
\begin{teo}\label{8teo2ver}
Let $\pi$ be a cuspidal representation, unramified at every place $v$, of $SO_{2n}(\A_\Q)$. Denote by $\pi_v$, the local representations. Suppose that Arthur's conjectures are verified for $\son$, i.e., the results announced in his article \cite{Ar2}, chapter $30$. Then
\begin{enumerate}
\item The Hecke matrix $t_{\pi_v}= \text{diag}(t_1, \dots, t_n,
  t_n^{-1}, \dots, t_1^{-1})$ of $\pi_v$ (the local component),
  satisfies in every finite place the estimates
\[
|t^{\pm 1}_i| \le p^{\frac{n}{2} -1 +\epsilon(n)},
\]
with $\epsilon(n)= \frac{1}{n^2+1}$.
\item At the archimedian place, we have also
\[
\pi_\infty \hookrightarrow ind_{B_{\son}}^{\son}(\chi_1,\dots,
\chi_n)\quad \text{(subquotient)}
\]
where $\chi_i:\R^\times \to \C^\times$, satisfies $\chi(x)=
|x|^{s_i}$, $|s_i| \le \frac{n}{2}-1 +\epsilon(n)$, $\epsilon(n)$ defined as above. We have thus the same type of bounds as for the finite places.
\end{enumerate}
\end{teo}

\begin{obs}
Finally, we remark, that we worked all the time over  $\Q$ for simplifying the  notations. The previous theorems remain true, under the same hypothesis, for any number field $F$.
\end{obs}

\bigskip

\begin{flushleft}
Octavio Paniagua Taboada \\
Mathematisches Institut \\
George-August-Universit\"at G\"ottingen\\
Bunsenstra\ss e 3-5\\
D-37073 Göttingen\\
Germany \\\bigskip
E-mail address:\\
{\tt paniagua@uni-math.gwdg.de}
\end{flushleft}


\begin{thebibliography}{99}

\bibitem[Art83]{Ar}
{\sc J. Arthur}, On some problems suggested by the trace formula, dans \emph{Lie Group Representations II}, Proceedings, University of Maryland 1982-1983, Herb, Kudla, Lipsman and Rosenberg eds, Springer Lecture Notes {\bf 1041}.

\bibitem[Art89]{Ar1}
{\sc J. Arthur}, Unipotent automorphic representations:conjectures, dans \emph{Orbites Unipotentes et Repr\'esentations II: Groupes $p$-adiques et r\'eels}, Asterisque, 171-172, 1989.


\bibitem[Art05]{Ar2}
{\sc J. Arthur}, An introduction to the trace formula, dans  Kottwitz, Robert E.
Arthur, James (ed.) et al., \emph{Harmonic analysis, the trace formula, and Shimura varieties}. Proceedings of the Clay Mathematics Institute 2003 summer school, Toronto, Canada, June 2--27, 2003. Providence, RI: American Mathematical Society (AMS). Clay Mathematics Proceedings 4, 393-522 (2005).

\bibitem[A-C89]{A-C}
{\sc J. Arthur, L. Clozel}, \emph{Simple Algebras, Base Change and the Advanced Theory of the Trace Formula},  Princeton University Press, Princeton, 1989.

\bibitem[Aub95]{Aub}
{\sc A.M. Aubert}, Dualité dans le groupe de Grothendieck de la catégorie des représentations lisses de longueur finie d'un groupe réductif $p$-adique. \emph{Trans. Amer. Math. Soc.},  {\bf 347} (1995),  no. 6, 2179--2189. 


\bibitem[BZ76]{BZ2}
{\sc I. N.  Berstein, A. V. Zelevinsky},  Representations of the group $GL(n,F)$, where $F$ is a local non-archimedian Field \emph{Uspekhi. Mat. Nauk.},   {\bf 31}, 3, 1976, 5-70 (in russian).


\bibitem[B-W99]{B-W}
{\sc A. Borel, N. Wallach}, \emph{Continuous Cohomology, Discrete Subgroups and Representations of Reductive Groups}, Mathematical Surveys and Monographs, vol $67$, A.M.S, 1999.


\bibitem[Clo07]{Clo}
{\sc L. Clozel},  Spectral Theory of Automorphic Forms. Automorphic Forms and Applications, 43-93. \emph{IAS/Park City Math. Ser.}, {\bf 12}, Amer. Math. Soc., Providence, RI, 2007.

\bibitem[Clo84]{Cl1}
{\sc L. Clozel},  Théorème d'Atiyah-Bott pour les variétés $p$-adiques et caractères des groupes réductifs. \emph{Mémoires de la Société Mathématique de France} Sér. 2, {\bf 15}, 1984, p. 39-64 

\bibitem[C-C]{C-C}
{\sc L. Clozel, G. Chenevier},  Corps de nombres peu ramifiés et formes automorphes autoduales.

\bibitem[C-U04]{C-U}
{\sc L. Clozel, E. Ullmo}, Équidistribution des points de Hecke. \emph{ Contributions to automorphic forms, geometry, and number theory},  193--254, Johns Hopkins Univ. Press, Baltimore, MD, 2004. 

\bibitem[vDj72]{vDj}
{\sc G. Van Dijk}, Computation de certain induced characters of $p$-adique groupes, \emph{Math. Ann.}, {\bf 199}, 1972, 229-240.

\bibitem[Dix69]{Dix}
{\sc J. Dixmier}, \emph{Les $C^*$-algebres et leurs representations}, Gauthier-Villars, Paris, 1969.


\bibitem[DHL92]{DHL}
{\sc W. Duke, R. Howe and J.-S. Li},  Estimating Hecke Eigenvalues of Siegel Modular Forms \emph{Duke Math. J.}, {\bf 6} (1992), 219-240.


\bibitem[How84]{How}
{\sc R. Howe}, Automorphic forms of low rank,  dans   \emph{Non-Commutative  Harmonic Analysis}, Lecture Notes in Math. {\bf 1041}, Springer, New York, 1984, 50-112.


\bibitem[J-S81]{J-S}
{\sc H. Jacquet, J. Shalika},  On Euler products and the classification of automorphic representations II, \emph{Amer. J. Math.}, {\bf 103}, 1981, 777-815.

\bibitem[Kat82]{Kat}
{\sc S.I. Kato}, Irreducibility of principal series representations for Hecke algebras of affine type, , \emph{J. Fac. Sci. Univ. Tokyo}, {\bf 28}, 929-944, 1982.

\bibitem[KSS03]{KSS}
{\sc H.H. Kim}, Functoriality for the exterior square of $GL_4$ and the symmetric fourth of $GL_2$, \emph{J. Amer . Math. Soc}, {\bf 16} (2003), no. 1, p. 139-183, With appendix 1 by Dinakar Ramakrishnan and appendix 2 by Kim and Sarnak.

\bibitem[Kna01]{Kna}
{\sc A. Knapp}, \emph{Representation Theory of Semisimple groups}. An overview based on examples, Reprint of the 1986 original. Third printing, and first Princeton Landmarks in Mathematics edition, 2001.

\bibitem[Kot84]{Kot}
{\sc R.E. Kottwitz}, Shimura varieties and twisted orbital integrals , \emph{Math. Ann.}, {\bf 269}, 287-300, 1984.

\bibitem[K-S99]{K-S}
{\sc R.E. Kottwitz, D. Shelstad}, Foundations of twisted endoscopy, \emph{Astérisque}, {\bf 255}, 1999.

\bibitem[Lab90]{Lab}
{\sc J.P. Labesse}, Fonctions \'el\'ementaires et lemme fondamental pour le changement de base stable \emph{Duke Math J.}, {\bf 61}, 519-530, 1990.

\bibitem[Lab84]{Lab2}
{\sc J.P. Labesse}, Cohomologie, $L$-groupes et fonctorialité, \emph{Compositio Math. J.}, {\bf 49}, 163-184, 1984.


\bibitem[Lan71]{L1}
{\sc R.P. Langlands}, \emph{Euler products}, Yale University Press, New Haven, 1971.

\bibitem[Lan76]{L2}
{\sc R.P. Langlands}, \emph{On the functional equations satisfied by Eisenstein series}, Lecture Notes in Mathematics, vol 544, Springer-Verlag, Berlin-Heidelberg-New York, 1976.

\bibitem[Lan79]{L3}
{\sc R.P. Langlands}, Stable conjugacy: definition and lemas,  \emph{Can. J. Math.}, {\bf 31}, 1979, p. 700-725.


\bibitem[L-S87]{L-S}
{\sc R.P. Langlands, D. Shelstad}, On the definition of transfer factors, \emph{Math. Ann.}, {\bf 278}, 1987, 219-271.

\bibitem[Li89]{Li}
{\sc J.-S. Li},  Singular Unitary Representations of Classical Groups, \emph{Invent. Math.}, {\bf 97} (1989), 237-255.


\bibitem[Mac78]{Mac}
{\sc G.W.  Mackey}, \emph{Unitary Group Representations in Physics, Probability and Number Theory}, The Benjamin/Cumming Publishing Company Inc, 1978.

\bibitem[M-W89]{MW2}
{\sc C. Moeglin et J.L. Waldspurger}, Le spectre résiduel de $GL(n)$ \emph{Ann. Scient. Éc. Norm. Sup}, $4$ série, {\bf 22}, 1989, 605-674.

\bibitem[M-W95]{M-W}
{\sc C. Moeglin et J.L. Waldspurger}, \emph{Spectral D\'ecomposition and Eisenstein Series}, Cambridge Tracts in Mathematics 113, Cambridge University Press, 1995.


\bibitem[Sil78]{Sil}
{\sc A.J. Silberger}, The Langlands quotient theorem for $p$-adic groups, \emph{Math. Ann.}, 263, 1978.

\bibitem[Sha81]{Sh1}
{\sc F. Shahidi}, On certain $L$-functions, \emph{Amer. J. Math}, 103, 1981.

\bibitem[Wal03]{Wal0}
{\sc J.L. Waldspurger}, La formule de Plancherel pour les groupes $p$-adiques. D'après Harish-Chandra, \emph{J. Inst. Math. Jussieu}, No.{\bf 2}, 235-333 (2003).

\bibitem[Wal07]{Wal}
{\sc J.L. Waldspurger}, Le groupe $\mathrm GL_N$ tordu, sur un corps $p$-adique. I., \emph{Duke Math. J.} 137, No. 2, 185-234 (2007).
\end{thebibliography}
\end{document}